\documentclass[12pt,english,sumlimits,reqno,oneside]{amsart}

\usepackage[T1]{fontenc}
\usepackage[utf8]{inputenc}
\usepackage{lmodern}

\usepackage{babel}

\usepackage[a4paper,margin=2.5cm]{geometry}

\usepackage{xifthen}
\usepackage{ifthenx}

\usepackage{xargs}

\usepackage{graphicx}
\usepackage[table,svgnames,x11names]{xcolor}

\usepackage{amssymb}
\usepackage{amsmath}
\usepackage{amsthm}
\usepackage{mathtools}

\usepackage{exscale}
\usepackage{relsize}
\usepackage{bbm}
\usepackage{bm}
\usepackage{amsbsy}
\usepackage{mathdots}

\usepackage{mathrsfs}

\usepackage{stmaryrd}

\usepackage{mdframed}

\usepackage[all]{xy}

\usepackage{fp}

\usepackage[section]{placeins}
\usepackage{float}

\usepackage{enumerate}

\newcounter{proof}
{\stepcounter{proof}\begin{proof}}%
{\end{proof}}%
\newcounter{proofstep}[proof]
{\refstepcounter{proofstep}\bigskip\par\noindent%
  \ifthenelse{\isempty{#1}}
    {\textit{Step \theproofstep. }}
    {\textit{#1.}}
  \noindent}%
{\par}%
\newcounter{proofstep-noskip}[proof]
\newenvironment{proofstep-noskip}[1][]%
{\refstepcounter{proofstep-noskip}
  \ifnum\the\value{proofstep-noskip}>1
    \bigskip\par\noindent
  \fi
  \ifthenelse{\isempty{#1}}
    {\textit{Step \theproofstep-noskip. }}
    {\textit{#1.}}
  \noindent}%
{\par}%
\newcounter{proofcase}[proof]
{\refstepcounter{proofcase}\bigskip\par\noindent%
  \ifthenelse{\isempty{#1}}
    {\textit{Case \theproofcase. }}
    {\textit{#1.}}
  \noindent}%
{\par}%

\newcounter{proofcase-noskip}[proof]
\newenvironment{proofcase-noskip}[1][]%
{\refstepcounter{proofcase-noskip}
  \ifnum\the\value{proofcase-noskip}>1
    \bigskip\par\noindent
  \fi
  \ifthenelse{\isempty{#1}}
    {\textit{Case \theproofcase-noskip. }}
    {\textit{#1.}}
  \noindent}%
{\par}%

\usepackage{varioref}
\usepackage{hyperref}
\hypersetup{
  colorlinks,
  allcolors=DarkBlue
}
\usepackage[nameinlink]{cleveref}

\theoremstyle{plain}
\newtheorem{thm}{Theorem}[section]
\newtheorem*{thm*}{Theorem}
\newtheorem{pro}[thm]{Proposition}

\newtheorem{lem}[thm]{Lemma}

\theoremstyle{definition}
\newtheorem{dfn}[thm]{Definition}
\theoremstyle{remark}

\newtheorem{rem}[thm]{Remark}

\numberwithin{equation}{section}

\AddToHook{env/pro/begin}{\crefalias{thm}{pro}} 
\AddToHook{env/cor/begin}{\crefalias{thm}{cor}}
\AddToHook{env/lem/begin}{\crefalias{thm}{lem}}
\AddToHook{env/question/begin}{\crefalias{thm}{question}}
\AddToHook{env/dfn/begin}{\crefalias{thm}{dfn}}
\AddToHook{env/ntn/begin}{\crefalias{thm}{ntn}}
\AddToHook{env/rem/begin}{\crefalias{thm}{rem}}
\AddToHook{env/clm/begin}{\crefalias{thm}{clm}}

\newcommandx{\textref}[2][1=]{\hyperref[#2]{#1\ref*{#2}}}
\newcommandx{\textrefp}[2][1=]{(\hyperref[#2]{#1\ref*{#2}})}

\newcommand{\dif}{\ensuremath{\, \mathrm d}}

\DeclareMathOperator{\supp}{supp}

\DeclareMathOperator{\spn}{span}

\newcommand{\vvvert}{{\vert\kern-0.25ex\vert\kern-0.25ex\vert}}

\usepackage{tikz}
\usetikzlibrary{trees}
\usepackage{tikz-cd}

\begin{document}

\title[Independent sums of Haar system Hardy spaces]%
{Factorization in independent sums of Haar system Hardy spaces}%

\author[K.~Konstantos]{Konstantinos Konstantos}%
\address{Konstantinos Konstantos, Department of Mathematics and Statistics, York University, 4700 Keele Street, Toronto, Ontario, M3J 1P3, Canada}%

\email{kostasko@yorku.ca}%

\author[T.~Speckhofer]{Thomas Speckhofer}%
\address{Thomas Speckhofer, Department of Mathematics and Statistics, York University, 4700 Keele Street, Toronto, Ontario, M3J 1P3, Canada}%
\email{tspeck@yorku.ca}%

\date{\today}%



\thanks{The first author was supported by NSERC Grant RGPIN-2021-03639.\\
\indent The second author was supported by the Austrian Science Fund~(FWF) projects I5231 and 10.55776/PAT9952223 and by NSERC Grant RGPIN-2021-03639.}

\begin{abstract}
  We introduce a generalization of the Bourgain-Rosenthal-Schechtman $R_{\omega}^p$ space: Let $Y$ be a Haar system Hardy space, i.e., a separable rearrangement-invariant function space on the unit interval or an associated Hardy space defined via the square function (such as dyadic~$H^1$). Then we define~$Y_{\omega}$ as the closed linear span in~$Y$ of independent distributional copies of the spaces~$Y_n$ of dyadic step functions at scale~$2^{-n}$. Combining finite-dimensional and infinite-dimensional techniques, we prove that the identity operator~$I$ on~$Y_{\omega}$ factors through every bounded linear operator~$T$ on~$Y_{\omega}$ which has large diagonal, and in general, the identity factors either through~$T$ or through~$I - T$.
\end{abstract}

\maketitle

\tableofcontents

\section{Introduction}
\label{sec:introduction}

In this paper, we prove factorization results for operators on \emph{independent sums of Haar system Hardy spaces}.
Our main motivation is to determine those bounded linear operators~$T$ on a given space~$X$ for which the identity operator~$I_X$ factors through~$T$, i.e., there exist bounded linear operators~$A,B$ such that the following diagram commutes:
\begin{equation*}
  \begin{tikzcd}
    X \arrow{r}{I_{X}} \arrow[swap]{d}{B} & X\\
    X \arrow{r}{T}& X \arrow[swap]{u}{A}
  \end{tikzcd}
\end{equation*}
Factorization problems of this type have played a crucial role in Banach space theory. In particular, they arose in the study of decompositions and complemented subspaces of classical Banach spaces (e.g., to identify primary Banach spaces) and in the theory of operator ideals. Recall that a Banach space~$X$ is primary if for every decomposition $X = E \oplus F$, either~$E$ or~$F$ is isomorphic to~$X$. This line of research was initiated in 1970 by J.~Lindenstrauss~\cite{MR0405074} and led to solutions of long-standing problems, such as the unconditional basic sequence problem, Banach's hyperplane problem, and the scalar-plus-compact problem~\cite{MR1201238,MR1315601,MR2784662,MR1999196}.

The classical sequence spaces $c_0$ and $\ell^p$, $1\le p<\infty$, were proved to be prime (which is stronger than primariness) in~1960 in A.~Pe{\l}czy\'{n}ski's seminal work~\cite{MR126145}. In~1975, the Lebesgue spaces $L^p$, $1\le p<\infty$, were shown to be primary by P.~Enflo and B.~Maurey~\cite{MR0397386}.
Subsequently, D.~Alspach, P.~Enflo and E.~Odell~\cite{MR0500076} gave an alternative proof and extended this result to separable rearrangement-invariant function spaces with
non-trivial Boyd indices, and P.~Enflo and
T.~W.~Starbird~\cite{MR0557491} proved the primariness of~$L^1$ using $E$-operators. Primariness has also been established for the dyadic Hardy space~$H^1$ and its dual $\mathrm{BMO}$ by P.~F.~X.~Müller~\cite{MR0931037,MR0955660}, for the bi-parameter spaces $L^p(L^q)$, $1 < p,q < \infty$ and~$L^p(X)$ (where $X$ has a symmetric basis) by M.~Capon~\cite{MR0687937,MR0688955}, for the bi-parameter dyadic Hardy space $H^1(\delta^2)$ by P.~F.~X.~Müller~\cite{MR1283008}, and for its dual~$\mathrm{BMO}(\delta^2)$ by R.~Lechner and P.~F.~X.~Müller~\cite{MR3436171}. Only recently, R.~Lechner, P.~Motakis, P.~F.~X.~Müller and Th.~Schlumprecht~\cite{MR4145794} have proved that the space~$L^1(L^p)$ is primary.

For most of the spaces mentioned above, primariness can be proved in two steps: In the first step, one proves that the given space~$X$ has the primary factorization property.
\begin{dfn}
We say that a Banach space~$X$ has the \emph{primary factorization property} if for every bounded linear operator~$T\colon X\to X$, the identity~$I_X$ factors either through~$T$ or through~$I_X - T$.
\end{dfn}
This implies that for any~$T$, the image of~$T$ or~$I_X - T$ has a complemented subspace isomorphic to~$X$. In the second step, one uses Pe{\l}czy\'{n}ski's decomposition method \cite{MR126145} (see also \cite[II.B.24]{MR1144277}) or a variation of this technique to obtain the desired isomorphism. In order for the decomposition method to be applicable, $X$ needs to satisfy Pe{\l}czy\'{n}ski's \emph{accordion property}, i.e., $X\sim \ell^p(X)$ for some $1\le p\le \infty$ or $X\sim c_0(X)$. Together with the primary factorization property, the accordion property implies primariness.

The primary factorization property is also related to the theory of operator ideals: For a Banach space~$X$, let $\mathcal{B}(X)$ denote the algebra of bounded linear operators on $X$. Then we define
\begin{equation*}
  \mathcal{M}_X = \{ T\in \mathcal{B}(X) : I_X \ne ATB\text{ for all } A, B\in \mathcal{B}(X) \}.
\end{equation*}
The set $\mathcal{M}_X$ is an ideal of $\mathcal{B}(X)$ if and only if it is closed under addition. It was observed by
D.~Dosev and W.~B.~Johnson~\cite{MR2586976} that if $\mathcal{M}_X$ is an ideal of $\mathcal{B}(X)$, then it is
automatically the \emph{largest} ideal (with respect to inclusion) of $\mathcal{B}(X)$. On the other hand, it is known that
$\mathcal{M}_X$ is closed under addition if and only if $X$ has the primary factorization property (see
\cite[Proposition~5.1]{MR2586976}).

In~\cite{MR4477014}, T.~Kania and R.~Lechner proved the primary factorization property for stopping-time Banach spaces. This was achieved by extending the general framework of \emph{strategically reproducible bases} established by R.~Lechner, P.~Motakis, P.~F.~X.~Müller and Th.~Schlumprecht \cite{MR4145794,MR4299595,MR4430957}, which involves winning strategies in infinite two-player games. Earlier, this framework was used to analyze operators on~$L^1$ with large diagonal with respect to the Haar basis:
Let~$X$ be a Banach space with a Schauder basis
$(e_j)_{j=1}^{\infty}$ and biorthogonal functionals $(e_j^{*})_{j=1}^{\infty}$. We say that a bounded linear operator $T\colon X\to X$ has \emph{large diagonal} with respect to~$(e_j)_{j=1}^{\infty}$ if its diagonal entries are uniformly bounded away from~$0$, i.e., $\inf_{j\in \mathbb{N}} |\langle e_j^{*}, Te_j \rangle| > 0$.
The study of operators with large diagonal goes back to
A.~D.~Andrew~\cite{MR0567081}, and they were explicitly studied
in~\cite{MR0890420,MR3861733,MR3819715,MR3794334,MR3990955,MR3910428,MR4460217,MR4145794}. In particular, in~\cite{MR4145794} it was proved that the Haar basis of~$L^1$ has the factorization property:
\begin{dfn}
  We say that a Schauder basis~$(e_j)_{j=1}^{\infty}$ of a Banach space~$X$ has the \emph{factorization property} if the identity~$I_X$ factors through every bounded linear operator~$T\colon X\to X$ which has large diagonal with respect to~$(e_j)_{j=1}^{\infty}$.
\end{dfn}
Subsequently, Kh.~V.~Navoyan~\cite{MR4635015} proved the factorization property for the Haar basis in separable rearrangement-invariant function spaces, assuming that the Haar basis is unconditional. More generally, R.~Lechner and the second author~\cite{MR4894823} developed a unified and systematic approach to prove the primary factorization property as well as the factorization property of the Haar basis in every \emph{Haar system Hardy space}~$Y$ (excluding only the Cantor space~$C(\Delta)$).
A Haar system Hardy space is defined as the completion of the linear span of the Haar system~$(h_I)_I$ under a rearrangement-invariant norm~$\|\cdot \|$ or under the associated square function norm $\|\cdot\|_{\circ}$ given by
\begin{equation*}
  \Bigl\| \sum_{I} a_I h_I \Bigr\|_{\circ}
  = \Bigl\| \Bigl(\sum_{I} a_I^2 h_I^2\Bigr)^{1/2} \Bigr\|.
\end{equation*}
In particular, this class contains~$L^p$ and dyadic~$H^p$, $1\le p<\infty$. The results in~\cite{MR4894823} are proved by performing a step-by-step reduction in which an arbitrary operator~$T$ is reduced to a constant multiple of the identity. Similar factorization results for multipliers on \emph{bi-parameter} Haar system Hardy spaces have recently been proved in~\cite{MR4816119}.

In~\cite{MR4839586}, the first author and P.~Motakis proved the factorization property as well as the primary factorization property for the Rosenthal~$X_{p,w}$ spaces and for the  Bourgain-Rosenthal-Schechtman $R_{\omega}^p$~space, $1<p<\infty$. The latter space is defined as an independent sum of finite-dimensional~$L^p_n$ spaces, and it was introduced in~\cite{MR632839} as part of an uncountable family of mutually nonisomorphic, complemented subspaces of~$L^p$ (for $1<p<\infty$ and $p\ne 2$). In the same spirit, P.~F.~X.~Müller and G.~Schechtman~\cite{MR1008719} constructed a new complemented subspace of the dyadic Hardy space~$H^1$ which is isomorphic to the independent sum of finite-dimensional $H^1_n$~spaces. More recently, M.~Rzeszut and M.~Wojciechowski~\cite{MR3648869} constructed a new idempotent Fourier multiplier on the Hardy space on the bidisc whose image also happens to be isomorphic to the independent sum of the spaces~$H^1_n$.

The proof of the factorization results for $R_\omega^p$ in~\cite{MR4839586} is again based on a step-by-step reduction, involving probabilistic and combinatorial arguments as well as the notion of orthogonal factors. In particular, the reduction to a diagonal operator uses a refinement of the probabilistic techniques introduced by R.~Lechner, which involve randomized block bases and were originally used to prove finite-dimensional, quantitative factorization results in Hardy spaces and $SL^{\infty}_n$~\cite{MR3990955,MR3910428}. In this finite-dimensional setting, the goal is to obtain a factorization of the identity on a smaller subspace through a given operator on a larger subspace while controlling the norms of all involved factors. These types of factorization results are closely related to the \emph{Restricted Invertibility Theorem} of J.~Bourgain and L.~Tzafriri~\cite{MR0890420}. The second author has recently extended the probabilistic techniques of R.~Lechner to prove finite-dimensional factorization results for all Haar system Hardy spaces~\cite{MR4884827}.

In this paper, we will combine these extended probabilistic techniques with infinite-dimensional methods to generalize the results of~\cite{MR4839586}, replacing the underlying space~$L^p$ with an arbitrary Haar system Hardy space~$Y$. We introduce the associated space~$Y_{\omega}$ as an independent sum of the finite-dimensional subspaces~$Y_n$ spanned by the first~$n+1$ levels of the Haar system. The space~$Y_{\omega}$ has a canonical basis $(h_I^n)_{(n,I)}$ indexed by pairs $(n,I)$ where $n$ is a natural number and $I$ is a dyadic interval. This construction is analogous to that of the space~$R_{\omega}^p$ and will be explained in detail in \Cref{sec:independent-sums}. In particular, this setting contains the independent sum of~$H^1_n$ spaces studied in~\cite{MR1008719,MR3648869}. Our main result is the following:
\begin{thm}\label{thm:main-result}
  Let~$Y$ be a Haar system Hardy space, and assume that the sequence of standard Rademacher functions is weakly null in~$Y$. Then $Y_{\omega}$ has the primary factorization property, and its canonical basis~$(h_I^n)_{(n,I)}$ has the factorization property.
  In particular, $\mathcal{M}_{Y_{\omega}}$ is the largest ideal of~$\mathcal{B}(Y_{\omega})$.
\end{thm}
We will prove the above statements in \Cref{thm:primary-and-pos-fact-prop} and \Cref{thm:fact-property}, where we also provide explicit factorization constants, i.e., upper bounds for the norms of the involved factors~$A$ and~$B$.
The condition involving Rademacher functions in \Cref{thm:main-result} only excludes the case $\|\cdot \|_Y\sim \|\cdot \|_{L^{\infty}}$ (see, e.g., \cite[Section~3.2]{speckhofer:2025}), where~$Y$ can be identified with~$C(\Delta)$, the space of continuous functions on the Cantor set.

The next section contains all necessary definitions and preliminary results on the space~$Y$. In \Cref{sec:independent-sums}, we introduce the space~$Y_\omega$ and prove some of its fundamental properties. In the remaining sections, we describe the various reduction steps which are needed to reduce an arbitrary operator on~$Y_{\omega}$ to a constant multiple of the identity, thus enabling us to prove our main result, \Cref{thm:main-result}.

\section{Preliminaries}
\label{sec:preliminaries}

In this section, we recall the notion of a \emph{Haar system Hardy space}, which was introduced in~\cite{MR4816119,MR4894823}, and we collect some fundamental properties of these spaces. Moreover, we discuss the notion of a \emph{faithful Haar system}, which will be used throughout the paper, and we introduce some terminology for describing factorizations and diagonals of operators.

\subsection{Notation and basic definitions}
\label{sec:notation}
We start by introducing some basic notation for dyadic intervals and the Haar system.
Let $\mathbb{N}_0 = \mathbb{N}\cup \{ 0 \}$, and let $\mathcal{D}$ denote the collection of all dyadic intervals in~$[0,1)$, i.e.,
\begin{equation*}
  \mathcal{D}
  = \Big\{\Big[\frac{i}{2^n},\frac{i+1}{2^n}\Big) :
  n\in\mathbb{N}_0,\ 0\leq i < 2^n\Big\}.
\end{equation*}
In addition, for each $n\in\mathbb{N}_0$, define
\begin{equation*}
  \mathcal{D}_n
  = \{I\in\mathcal{D} : |I| = 2^{-n}\} = \Big\{ \Big[ \frac{i}{2^{n}}, \frac{i+1}{2^n} \Big):\; 0\leq i < 2^n \Big\}
  \end{equation*}
  and
\begin{equation*}
  \mathcal{D}_{\le n}
  = \bigcup_{k=0}^n \mathcal{D}_k,
  \quad
  \mathcal{D}_{<n} = \bigcup_{k=0}^{n-1}\mathcal{D}_k,
\end{equation*}
as well as
\begin{equation*}
  \mathcal{D}_{\omega} = \bigl\{ (n,I) : n\in \mathbb{N}_0,\, I\in \mathcal{D}_{\le n} \bigr\}.
\end{equation*}
The set $\mathcal{D}_{\omega}$ will be used to enumerate the basis of $Y_{\omega}$ (see \Cref{dfn:independent-sum}). It can be visualized as disjoint copies of the sets $\mathcal{D}_{\leq n}$, $n \in \mathbb{N}_{0}$, which form finite dyadic trees. This is illustrated in the following diagram for $n = 0,1,2$:
\begin{center}
\begin{tikzpicture}[
  level distance=1.2cm,
  sibling distance=2.5cm,
  every node/.style={circle,fill,inner sep=0pt,minimum size=2mm},
  edge from parent/.style={draw,thick},
  level 1/.style={sibling distance=3cm},
  level 2/.style={sibling distance=2cm},
  ]

\node[label={[xshift=4pt]right:$0,[0,1)$}] (A) at (-4,0) {};

\node[label={[xshift=4pt]right:$1,[0,1)$}] (B) at (0,0) {}
  child {node[label={[yshift=12pt]below:$1,[0,\frac12)$}] {}}
  child {node[label={[yshift=12pt]below:$1,[\frac12,1)$}] {}};

\node[label={[xshift=4pt]right:$2,[0,1)$}] (C) at (6,0) {}
  child {node[label={[xshift=4pt]right:$2,[0,\frac12)$}] {}
    child {node[label={[xshift=-6pt,yshift=12pt]below:$2,[0,\frac14)$}] {}}
    child {node[label={[xshift=-12pt,yshift=12pt]below:$2,[\frac14,\frac12)$}] {}}
  }
  child {node[label={[xshift=4pt]right:$2,[\frac12,1)$}] {}
    child {node[label={[xshift=12pt,yshift=12pt]below:$2,[\frac12,\frac34)$}] {}}
    child {node[label={[xshift=6pt,yshift=12pt]below:$2,[\frac34,1)$}] {}}
  };
\end{tikzpicture}
\end{center}

If $I\in \mathcal{D}$ is a dyadic interval, then we denote by~$I^+$ the left half of~$I$ and by~$I^-$ the right half of~$I$ (both are elements of~$\mathcal{D}$). If we use the symbol~$\pm$ multiple times in
an equation, we mean either always~$+$ or always~$-$. For any
subcollection $\mathcal{B}\subset\mathcal{D}$, we put $\mathcal{B}^* = \bigcup_{I\in \mathcal{B}}I$.

Next, we define the bijective function $\iota\colon\mathcal{D}\to\mathbb{N}$ by
\begin{equation*}
  \Big[\frac{i}{2^n},\frac{i+1}{2^n}\Big)
  \overset{\iota}{\mapsto} 2^n + i.
\end{equation*}
We also define a linear order on~$\mathcal{D}_{\omega}$ by setting $(n,I) \prec (m,J)$ if and only if $n < m$ or ($n = m$ and $\iota(I) < \iota(J)$). Then $(\mathcal{D}_{\omega},\prec)$ is order isomorphic to~$(\mathbb{N},<)$. Whenever we write an infinite series $\sum_{(n,I)\in \mathcal{D}_{\omega}}$, we mean that the order of summation is determined by this order isomorphism.

The Haar system $(h_I)_{I\in\mathcal{D}}$ is defined as
\begin{equation*}
  h_I
  = \chi_{I^+} - \chi_{I^-},
  \qquad I\in\mathcal{D},
\end{equation*}
where $\chi_A$ denotes the characteristic function of a subset $A\subset [0,1)$.  We additionally
put $h_\varnothing = \chi_{[0,1)}$ and $\mathcal{D}^+ = \mathcal{D}\cup\{\varnothing\}$ as well as
$\iota(\varnothing) = 0$.  Recall that in the linear order induced by $\iota$, the Haar system
$(h_I)_{I\in \mathcal{D}^+}$, is a monotone Schauder basis
of $L^p$, $1\le p<\infty$ (and unconditional if $1<p<\infty$).
For $x = \sum_{I\in \mathcal{D}^+} a_Ih_I\in L^1$, by the \emph{Haar support} of~$x$ we mean the
set of all $I\in \mathcal{D}^+$ such that $a_I\ne 0$. Moreover, we write $\operatorname{supp} x = \{ t\in [0,1) : x(t) \ne 0 \}$, and for $\alpha\in \mathbb{R}$, we use the notation $[x\in \alpha] = \{ t\in [0,1) : x(t) = \alpha \}$. We only consider Banach spaces over the real numbers.

\subsection{Haar system Hardy spaces}
\label{sec:haar-system-hardy-spaces}

In this subsection we recall the definition of a Haar system Hardy space. To this end, we first need the notion of a \emph{Haar system space}, which was introduced in~\cite[Section~2.4]{MR4430957}.
\begin{dfn}\label{dfn:HS-1d}
  A \emph{Haar system space $X$} is the completion of
  $H := \spn \{ h_I \}_{I\in\mathcal{D}^+} = \spn\{\chi_I\}_{I\in\mathcal{D}}$ under a norm
  $\|\cdot\|_X$ that satisfies the following properties:
  \begin{enumerate}[(i)]
    \item\label{dfn:HS-1d:1} If $f$, $g$ are in $H$ and $|f|$, $|g|$ have the same distribution,
          then $\|f\|_X = \|g\|_X$.
    \item\label{dfn:HS-1d:2} $\|\chi_{[0,1)}\|_X = 1$.
  \end{enumerate}
  We denote the class of Haar system spaces by $\mathcal{H}(\delta)$.  Given
  $X\in\mathcal{H}(\delta)$, we define the closed subspace $X_0\subset X$ as the closure of
  $H_0 := \spn\{ h_I \}_{I\in \mathcal{D}}$ in $X$.  We denote the class of these subspaces by
  $\mathcal{H}_0(\delta)$.
\end{dfn}
Apart from the spaces $L^p$, $1\le p<\infty$, and the closure of $H$ in $L^{\infty}$, the class
$\mathcal{H}(\delta)$ includes all rearrangement-invariant function spaces~$X$ on $[0,1)$ (e.g.,
Orlicz function spaces) in which the linear span~$H$ of the Haar system $(h_I)_{I\in \mathcal{D}^+}$ is
dense. If $H$ is not dense in a given rearrangement-invariant function space~$X$, then its closure in~$X$ is a Haar system space.

The following proposition contains some basic results on Haar system spaces which will be used frequently throughout the paper. For proofs, we refer to~\cite[Section~2.4]{MR4430957} and~\cite{MR4894823} (see also~\cite[Chapter~2]{speckhofer:2025} for a comprehensive exposition containing all results and proofs).

\begin{pro}\label{pro:HS-1d}
  Let $X\in\mathcal{H}(\delta)$ be a Haar system space. Then the following hold:
  \begin{enumerate}[(i)]
    \item\label{pro:HS-1d:i} For every $f\in H = \spn\{h_I\}_{\in\mathcal{D}^+}$, we have
          $\|f\|_{L^1}\leq \|f\|_X\leq \|f\|_{L^\infty}$.
    \item\label{pro:HS-1d:v} For all $f,g\in H$ with $|f|\leq |g|$, we have $\|f\|_X \leq \|g\|_X$.
    \item\label{pro:HS-1d:ii} The Haar system $(h_I)_{I\in \mathcal{D}^+}$, in the usual linear
          order, is a monotone Schauder basis of~$X$.
    \item \label{pro:HS-1d:iii} $H$ naturally coincides with a subspace of $X^{*}$ (where $f\in H$ acts on $g\in H$ by $\langle f,g \rangle = \int fg$), and its closure in $X^{*}$ is also a Haar system space.
    \item For every $I\in \mathcal{D}$, we have $\|h_I\|_X\|h_I\|_{X^{*}} = |I|$.
  \end{enumerate}
\end{pro}

We will also need the following lemma, which states that certain conditional expectations are bounded in every Haar system space. A proof can be found in~\cite[Lemma~4.3]{MR4894823}.
\begin{lem} \label{lem:cond-expect} Let $X\in\mathcal{H}(\delta)$ and let $\mathcal{F}$ be a $\sigma$-algebra generated by a partition $\{A_{i}  \}_{i=1}^{n}$ of~$[0,1)$, where for every $1 \leq i \leq n$, the set $A_i$ is a finite union of dyadic intervals. Then
\begin{align*}
\| \mathbb{E}^{\mathcal{F}}x \|_{X} \leq \| x \|_{X},\qquad x \in H.
\end{align*}
\end{lem}

Next, by $(r_n)_{n=0}^{\infty}$, we denote the sequence of standard Rademacher functions, i.e.,
\begin{equation*}
  r_n = \sum_{I\in \mathcal{D}_n} h_I,\qquad n\in \mathbb{N}_0.
\end{equation*}
Using the sequence of Rademacher functions, we define the set
\begin{equation*}
  \mathcal{R} = \{ (r_{\iota(I)})_{I\in \mathcal{D}^+}, (r_0)_{I\in \mathcal{D}^+} \}.
\end{equation*}
Hence, if $\mathbf{r} = (r_I)_{I\in \mathcal{D}^+}\in \mathcal{R}$, then $\mathbf{r}$ is either an
independent sequence of $\pm 1$-valued random variables or a constant
sequence, indexed by dyadic intervals.

We can now give the definition of a Haar system Hardy space (cf.~\cite{MR4816119,MR4894823}).
\begin{dfn}\label{dfn:haar-system-hardy-space}
  For $X\in\mathcal{H}(\delta)$ and
  $\mathbf{r} = (r_I)_{I\in\mathcal{D}^+}\in\mathcal{R}$, we define the \emph{(one-parameter) Haar
    system Hardy space $X(\mathbf{r})$} as the completion of
  $H = \spn\{ h_I \}_{I\in \mathcal{D}^+}$ under the norm $\|\cdot\|_{X(\mathbf{r})}$ given by
  \begin{equation*}
    \Bigl\| \sum_{I\in\mathcal{D}^+} a_I h_I\Bigr\|_{X(\mathbf{r})}
    = \Bigl\|
    s\mapsto \int_0^1 \Bigl|
    \sum_{I\in\mathcal{D}^+} r_I(u) a_I h_I(s)
    \Bigr| \dif u
    \Bigr\|_X.
  \end{equation*}
  We denote the class of one-parameter Haar system Hardy spaces by $\mathcal{HH}(\delta)$.
  Moreover, given $X(\mathbf{r})\in\mathcal{HH}(\delta)$, we define the subspace
  $X_0(\mathbf{r})$ as the closure of $H_0 = \operatorname{span}\{ h_I \}_{I\in \mathcal{D}}$ in~$X(\mathbf{r})$.
  We denote the class of these subspaces by $\mathcal{H}\mathcal{H}_0(\delta)$.
  For notational convenience, we will also refer to the subspaces $X_0(\mathbf{r})$ as Haar system Hardy spaces, and we will usually denote them by~$Y$, suppressing the parameters~$X$ and~$\mathbf{r}$.
  We will also work with the finite-dimensional subspaces $Y_n := \operatorname{span}\{ h_I \}_{I\in \mathcal{D}_{\le n}}$, $n\in \mathbb{N}_0$, of any Haar system Hardy space~$Y$.
\end{dfn}
Note that if $r_I = r_0$ for all $I\in\mathcal{D}^+$, then we have $\|\cdot \|_{X(\mathbf{r})} = \|\cdot \|_X$ and thus $X(\mathbf{r}) = X$, so the class~$\mathcal{HH}(\delta)$ is an extension of $\mathcal{H}(\delta)$. Moreover, if $\mathbf{r}$ is independent, then by Khintchine's inequality, the norm on $X(\mathbf{r})$ can be expressed in terms of the square function:
\begin{equation*}
  \Bigl\| \sum_{I\in \mathcal{D}^+}a_Ih_I \Bigr\|_{X(\mathbf{r})} \sim \Bigl\| \Bigl( \sum_{I\in \mathcal{D}^+} a_I^2h_I^2 \Bigr)^{1/2} \Bigr\|_X.
\end{equation*}
Hence, for $X = L^1$ and an independent Rademacher sequence~$\mathbf{r}$, we obtain the dyadic Hardy space~$H^1$, equipped with an equivalent norm.

In the next proposition, we collect some basic properties of Haar system Hardy spaces.
For proofs, we again refer to~\cite{MR4894823} and~\cite[Chapter~2]{speckhofer:2025}.

\begin{pro}\label{pro:HSHS-1d}
  Let $X(\mathbf{r})\in\mathcal{HH}(\delta)$, and put $Y = X_0(\mathbf{r})$. Then the following hold:
  \begin{enumerate}[(i)]
    \item\label{pro:HSHS-1d:i} The Haar system $(h_I)_{I\in \mathcal{D}}$, in the usual linear
          order, is a monotone Schauder basis of~$Y$, and it is $1$-unconditional if $\mathbf{r}$ is independent.
    \item $H_0 = \operatorname{span}\{ h_I \}_{I\in \mathcal{D}}$ naturally coincides with a subspace of $Y^{*}$, and if $\mathbf{r}$ is independent, then $(h_I)_{I\in \mathcal{D}}$ is $1$-unconditional in $Y^{*}$.
    \item\label{pro:HSHS-1d:iii} For every $I\in \mathcal{D}$, we have $\|h_I\|_Y\|h_I\|_{Y^{*}} = |I|$. In particular, we have $\|h_I\|_Y = \|h_I\|_X$ and $\|h_I\|_{Y^{*}} = \|h_I\|_{X^{*}}$.
    \item\label{pro:HSHS-1d:disjointly-supp} If $f$ is a finite linear combination of disjointly supported Haar functions, then we have $\|f\|_{X(\mathbf{r})} = \|f\|_X$.
    \item\label{pro:HSHS-1d:single-layer} If $x = \sum_{I\in \mathcal{D}} a_Ih_I\in H_0$, then for every $k\in \mathbb{N}_0$, we have $\|\sum_{I\in \mathcal{D}_k}a_Ih_I\|_Y\le \|x\|_Y$. Moreover, $\|a_Ih_I\|_Y\le \|x\|_Y$ for all $I\in \mathcal{D}$.
    \item\label{pro:HSHS-1d:v} If $\mathcal{B}$ is a finite collection of pairwise disjoint dyadic intervals and $(\theta_K)_{K\in \mathcal{B}}\in \{ \pm 1 \}^{\mathcal{B}}$, then $\|\sum_{K\in \mathcal{B}}\theta_Kh_K\|_{Y^{*}}\le 1$.
  \end{enumerate}
\end{pro}

\subsection{Faithful Haar systems}

In this subsection we discuss the notion of a \emph{faithful Haar system}. This is a system of functions which are blocks of the Haar system and have a structure analogous to that of the original Haar system. In particular, their supports exhibit the same intersection pattern as the supports of the original Haar functions.
The term \emph{faithful Haar system} was first introduced in~\cite{MR4430957}, but the construction had already been utilized much earlier---see, e.g.,~\cite[p.~51]{MR0402915}.
We will also consider \emph{almost faithful Haar systems}, i.e., perturbations of faithful Haar systems which permit small gaps between the supports of the functions. This construction is due to J.~L.~B.~Gamlen and R.~J.~Gaudet~\cite{MR0328575} and has played a key role the classical works of Enflo-Maurey~\cite{MR0397386}, Alspach-Enflo-Odell~\cite{MR0500076}, Maurey~\cite{MR0586594} and others. A well-known variant of this notion is given by the \emph{compatibility conditions} of P.~W.~Jones~\cite{MR0796906}, which ensure that such a system is complemented in~BMO (see also \cite[p.~105]{MR2157745} and the generalizations in~\cite{MR3819715,MR0955660,MR1283008}).

In this paper, on the one hand, we will use finite faithful Haar systems to construct the space~$Y_\omega$, and on the other hand, we will utilize this concept in the proofs of our main results.

\begin{dfn}\label{dfn:faithful}
  Let $n\in \mathbb{N}_0$, and for every $I\in \mathcal{D}_{\le n}$, let $\mathcal{B}_I$ be a non-empty finite subcollection of $\mathcal{D}$. Moreover, let $(\theta_K)_{K\in \mathcal{D}}$ be a family of signs, and put $\tilde{h}_I = \sum_{K\in \mathcal{B}_I} \theta_Kh_K$ for $I\in \mathcal{D}_{\le n}$. We say that $(\tilde{h}_I)_{I\in \mathcal{D}_{\le n}}$ is a \emph{finite almost faithful Haar system} if the following conditions are satisfied:
\begin{enumerate}[(i)]
  \item\label{dfn:faithful:i} For every $I\in \mathcal{D}_{\le n}$, the collection $\mathcal{B}_I$ consists of pairwise disjoint dyadic intervals, and we have $\mathcal{B}_I\cap \mathcal{B}_J = \emptyset$ for all $I\ne J\in \mathcal{D}_{\le n}$.
  \item\label{dfn:faithful:ii} For every $I\in \mathcal{D}_{<n}$, we have $\mathcal{B}_{I^{\pm }}^{*}\subset [ \tilde{h}_I = \pm 1 ]$.
\end{enumerate}
If $\mathcal{B}_{[0,1)}^{*} = [0,1)$ and $\mathcal{B}_{I^{\pm }}^{*} = [ \tilde{h}_I = \pm 1 ]$ for all $I\in \mathcal{D}_{<n}$, then we say that $(\tilde{h}_I)_{I\in \mathcal{D}_{\le n}}$ is \emph{faithful}, and in this case, we usually denote the system by $(\hat{h}_I)_{I\in \mathcal{D}_{\le n}}$. If $k_0<k_1<\dots<k_n$ is a strictly increasing sequence of integers such that $\mathcal{B}_I\subset \mathcal{D}_{k_i}$ for all $I\in \mathcal{D}_i$, $0\le i\le n$, then we say that $(\tilde{h}_I)_{I\in \mathcal{D}_{\le n}}$ is an almost faithful Haar system \emph{with frequencies $k_0,\dots,k_n$}. An example of a finite faithful Haar system is given in \Cref{fig:faithful-haar-system}.
\end{dfn}

Note that every finite faithful Haar system~$(\hat{h}_I)_{I\in \mathcal{D}_{\le n}}$ is distributionally equivalent to the standard Haar system~$(h_I)_{I\in \mathcal{D}_{\le n}}$, i.e., for every family of scalars~$(a_I)_{I\in \mathcal{D}_{\le n}}$, the functions $\sum_{I\in \mathcal{D}_{\le n}} a_I \hat{h}_I$ and $\sum_{I\in \mathcal{D}_{\le n}} a_Ih_I$ have the same distribution. In particular, $|\operatorname{supp} \hat{h}_I| = |I|$ for all~$I\in \mathcal{D}_{\le n}$.

We will later use the concept of faithful Haar systems to construct independent distributional copies of the standard Haar system, which will in turn be used to define the space~$Y_{\omega}$. Moreover, to prove our main results, we will frequently construct a sequence of independent finite (almost) faithful Haar systems and combine them to obtain an \emph{infinite} system in~$Y_\omega$, spanning a subspace on which a given operator has certain desired properties. This leads to the notion of an \emph{(almost) faithful Haar system in~$Y_\omega$}, which will be introduced in \Cref{sec:embeddings-projections}.

\begin{figure}[t]
  \centering \includegraphics{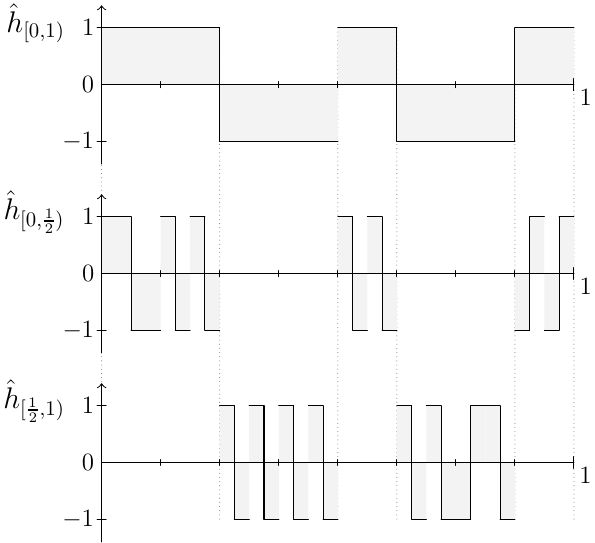}
  \caption{A finite faithful Haar system}
  \label{fig:faithful-haar-system}
\end{figure}

Finally, we describe a basic method to construct a finite faithful Haar system with given frequencies, which will be used in the definition of~$Y_\omega$ (\Cref{dfn:independent-sum}).
\begin{rem}\label{rem:faithful-from-rademachers}
  Suppose that we are given a strictly increasing finite sequence of non-negative integers $k_0 < k_1 < \dots < k_n$. Then we can inductively define a finite faithful Haar system with frequencies $k_0,\dots,k_n$ as follows: Define $\mathcal{B}_{[0,1)} = \mathcal{D}_{k_0}$ and $\hat{h}_{[0,1)} = r_{k_0}$, and for all $0\le j < n$ and $I\in \mathcal{D}_j$, put
  \begin{equation*}
    \mathcal{B}_{I^{\pm }} = \bigl\{ K\in \mathcal{D}_{k_{j+1}} : K\subset [\hat{h}_I = \pm 1] \bigr\}\qquad \text{and} \qquad
    \hat{h}_{I^{\pm }} = \sum_{K\in \mathcal{B}_{I^{\pm }}} h_K = r_{k_{j+1}}\cdot \chi_{[\hat{h}_I = \pm 1]}.
  \end{equation*}
  Note that by induction, we see that all functions $\hat{h}_I$, $I\in \mathcal{D}_{\le n}$, are measurable with respect to the $\sigma$-algebra generated by $r_{k_0},r_{k_1},\dots,r_{k_n}$.
\end{rem}

\subsection{Operators and factorization} In this subsection we introduce some terminology for stating factorization results and describing properties of linear operators and their diagonals (cf.~\cite{MR4430957,MR4894823}).

\begin{dfn}\label{dfn:factorization-modes}
  Let $X$ and $Y$ denote Banach spaces. Let $S\colon X\to X$ and $T\colon Y\to Y$ be bounded linear operators, and let
  $C, \eta \ge 0$.
  \begin{enumerate}[(i)]
    \item\label{enu:dfn:factorization-modes:a} We say that \emph{$S$ factors through $T$ with
          constant $C$ and error $\eta$} if there exist bounded linear operators $B\colon X\to Y$ and $A\colon Y\to X$ with
          $\|A\|\cdot \|B\|\leq C$ such that $\|S - ATB\|\leq \eta$.
    \item\label{enu:dfn:factorization-modes:a2} If~\eqref{enu:dfn:factorization-modes:a} holds and
          we additionally have $AB = I_X$, then we say that $S$ \emph{projectionally} factors
          through $T$ with constant $C$ and error $\eta$.
    \item We say that $S$ factors through~$T$ \emph{with constant $C^+$} if for every $\eta > 0$, $S$ factors through~$T$ with constant~$C + \eta$.
  \end{enumerate}
  If we omit the phrase ``with error $\eta$'' in~\eqref{enu:dfn:factorization-modes:a}
  or~\eqref{enu:dfn:factorization-modes:a2}, then we take that to mean that the error is~$0$.
\end{dfn}

The following observation will play an important role in the proofs of our main results. It goes back to~\cite{MR4430957}, where projectional factorizations were first utilized.

\begin{rem}\label{rem:factorization-identity-minus-T}
  Let $X,Y$ and $S,T$ be as above and suppose that $S$ \emph{projectionally} factors through~$T$ with constant~$C$ and error~$\eta$.  Then $I_X - S$ projectionally factors through $I_Y - T$ with the same constant and error.
\end{rem}

Moreover, the existence of a projectional factorization is a transitive property:
\begin{rem}\label{rem:transitivity}
  Let $R,S,T$ denote bounded linear operators between Banach spaces and suppose that~$R$ (projectionally) factors through~$S$ with constant~$C_1$ and error~$\eta_1$, and that~$S$ (projectionally) factors through~$T$ with constant~$C_2$ and error~$\eta_2$. In \cite[Proposition 2.3]{MR4430957}, it was observed that in this case,~$R$ (projectionally) factors through~$T$ with constant~$C_1C_2$ and error~$\eta_1 + C_1 \eta_2$.
\end{rem}

Now let $X$ be a Banach space with a Schauder basis $(e_n)_{n\in \mathbb{N}}$ and associated coordinate functionals~$(e_n^{*})_{n\in \mathbb{N}}$. We can identify any bounded linear operator $T \colon X \to X$ with its matrix representation $(\langle e_n^{*}, T e_m \rangle)_{n,m \in \mathbb{N}}$.
In particular, this allows us to define \emph{diagonal operators} (i.e., operators whose off-diagonal entries vanish) and operators with \emph{large diagonal} (i.e., whose diagonal entries $\langle e_n^{*}, Te_n \rangle$ are uniformly bounded away from zero).
\begin{dfn}\label{dfn:large-diagonal}
  Let $X$ be a Banach space with a Schauder basis (or a finite basis)~$(e_n)_n$ and associated coordinate functionals~$(e_n^{*})_n$. Moreover, let $\delta > 0$, and let $T\colon X\to X$ be a bounded linear operator.
  We say that
  \begin{enumerate}[(i)]
    \item $T$ has \emph{$\delta$-large diagonal with respect to~$(e_n)_n$} if
          $|\langle e_n^{*}, Te_n \rangle| \ge \delta $ for all~$n$.
    \item $T$ has $\delta$-large \emph{positive} diagonal with respect to~$(e_n)_n$ if
          $\langle e_n^{*}, Te_n \rangle \ge \delta$ for all~$n$.
    \item $T$ has \emph{large (positive) diagonal} if it has $\delta$-large (positive) diagonal for some~$\delta > 0$.
    \item $T$ is a \emph{diagonal operator} with respect to~$(e_n)_n$ if $\langle e_m^{*}, T e_n \rangle = 0$ whenever $m\ne n$.
  \end{enumerate}
  If it is clear from the context which basis~$(e_n)_n$ of~$X$ is meant, we sometimes just call~$T$ a diagonal operator or say that $T$ has large diagonal without specifying the basis.
  A diagonal operator~$D$ with respect to the Haar system~$(h_I)_{I\in \mathcal{D}}$ is called a \emph{Haar multiplier}, and its \emph{entries} $(d_I)_{I\in \mathcal{D}}$ are given by $d_I = \langle h_I, D h_I \rangle/|I|$, $I\in \mathcal{D}$.
  More generally, we define the set of \emph{diagonal entries} of any bounded linear operator~$T\colon X\to X$ by
  \begin{equation*}
    \operatorname{diag}(T) = \{ \langle e_n^{*}, Te_n \rangle : n\in \mathbb{N} \}
  \end{equation*}
  (with the obvious modification if the basis~$(e_n)_n$ is finite). By~$\mathcal{A}(\operatorname{diag}(T))$, we denote the set of all averages of finitely many elements of~$\operatorname{diag}(T)$ (where we allow repetitions).
\end{dfn}

\section{Independent sums of Haar system Hardy spaces}
\label{sec:independent-sums}

In this section, we introduce the space~$Y_{\omega}$ associated with a Haar system Hardy space~$Y$. It will be defined as the closed linear span in~$Y$ of independent distributional copies of the finite-dimensional spaces $Y_n = \operatorname{span}\{ h_I \}_{I\in \mathcal{D}_{\le n}}$, $n \in \mathbb{N}_{0}$. Moreover, we prove some basic properties of the space~$Y_{\omega}$, and in particular, we obtain a convenient formula for its norm (\Cref{lem:independent-sum-norm}).

\begin{dfn}\label{dfn:independent-sum}
Let $Y \in \mathcal{HH}_0(\delta)$ be a Haar system Hardy space. Moreover, let $(\tau_n)_{n=0}^{\infty}$ be a sequence of pairwise disjoint subsets of~$\mathbb{N}_0$ such that for every~$n$, the set~$\tau_n$ has cardinality~$n + 1$. Then, by \Cref{rem:faithful-from-rademachers}, for every $n\in \mathbb{N}_0$, we can construct a finite faithful Haar system $(h_I^n)_{I\in \mathcal{D}_{\le n}}$ whose frequencies are the elements of $\tau_n$ in increasing order. We define the independent sum~$Y_{\omega}$ as the closure of~$H_{\omega} := \operatorname{span}\{ h_I^n \}_{(n,I)\in \mathcal{D}_{\omega}}$ in~$Y$.
\end{dfn}

By the last statement of \Cref{rem:faithful-from-rademachers}, the $\sigma$-algebras $\mathcal{H}_n := \sigma(\{ h_I^n \}_{I\in \mathcal{D}_{\le n}})$, $n\in \mathbb{N}_0$, are independent. Thus, the systems $(h_I^n)_{I\in \mathcal{D}_{\le n}}$, $n\in \mathbb{N}_0$, form independent distributional copies of $(h_I)_{I\in \mathcal{D}_{\le n}}$, $n\in \mathbb{N}_0$, where $(h_I)_{I\in \mathcal{D}}$ is the standard Haar system.
As special cases, if the underlying space is~$L^p$, $1<p<\infty$, we obtain the Bourgain-Rosenthal-Schechtman~$R_{\omega}^p$ space~\cite{MR632839}, and in dyadic~$H^1$, we obtain the independent sum of the spaces~$H^1_n$ studied in~\cite{MR1008719,MR3648869}.
The first few elements of the system~$(h_I^n)_{(n,I)\in \mathcal{D}_{\omega}}$ are shown in \Cref{fig:basis-hIn} for the simple choice $\tau_0 = \{ 0 \}$, $\tau_1 = \{ 1,2 \}$, $\tau_2 = \{ 3,4,5 \}$.
\begin{figure}[t]
  \centering \includegraphics{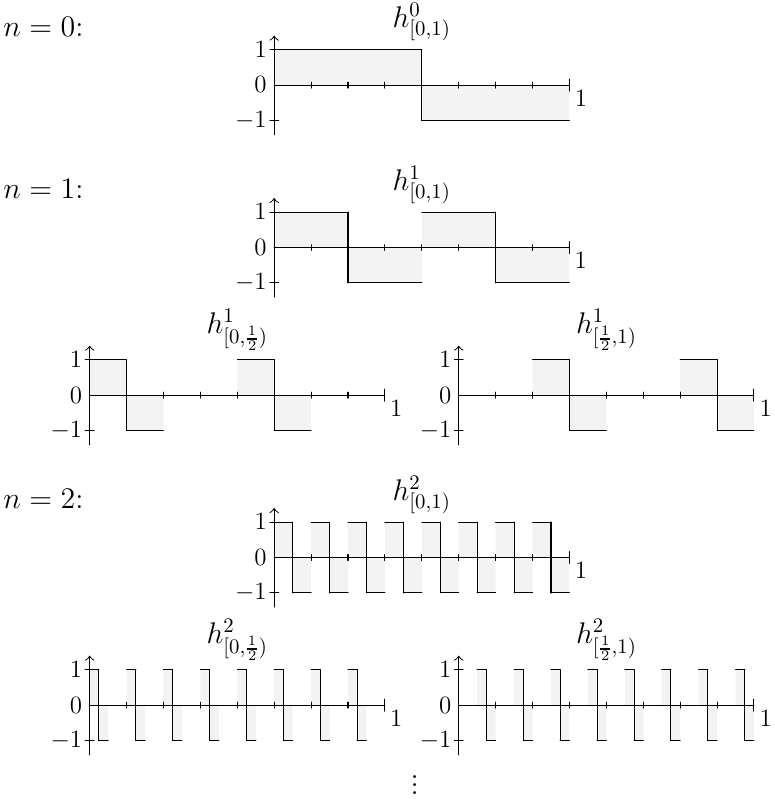}
  \caption{The functions $h_I^n$}
  \label{fig:basis-hIn}
\end{figure}

From now on, whenever we have a space~$Y_\omega$ constructed from $Y = X_0(\mathbf{r})$ for some Haar system space~$X$ and some sequence $\mathbf{r} = (r_I)_{I\in \mathcal{D}} \in \mathcal{R}$, we perform a relabeling of~$\mathbf{r}$ to obtain an associated family~$(r_I^n)_{(n,I)\in \mathcal{D}_\omega}$ indexed by pairs~$(n,I)$. Hence, if~$\mathbf{r}$ is constant, then $(r_I^n)_{(n,I)\in \mathcal{D}_\omega}$ is constant, and if~$\mathbf{r}$ is independent, then $(r_I^n)_{(n,I)\in \mathcal{D}_\omega}$ is independent.
Using this notation, the norm in $Y_{\omega}$ can be computed as follows:
\begin{lem}\label{lem:independent-sum-norm}
  Let $Y = X_0(\mathbf{r})$ be a Haar system Hardy space, and let $(h_I^n)_{(n,I)\in \mathcal{D}_{\omega}}$ be constructed as in \Cref{dfn:independent-sum}. Then for every $x = \sum_{(n,I)\in \mathcal{D}_{\omega}} a_I^nh_I^n\in H_{\omega}$, we have
\begin{align*}
  \| x \|_Y
  &= \Big\| s \mapsto \int_0^1 \Big| \sum_{(n,I)\in \mathcal{D}_{\omega}} r_I^n(u) a_I^nh_I^n(s) \Big|\, \mathrm{d}u \Big\|_{X}.
\end{align*}
\end{lem}
\begin{proof}
  Write $h_I^n = \sum_{K\in \mathcal{C}_I^n} h_K$ for $(n,I)\in \mathcal{D}_{\omega}$, where $\mathcal{C}_I^n$ is a finite collection of pairwise disjoint dyadic intervals and $\mathcal{C}_I^n \cap \mathcal{C}_J^m = \emptyset$ for all $(n,I) \ne (m,J)\in \mathcal{D}_{\omega}$. Then we have
  \begin{equation}\label{eq:17}
    \|x\|_Y = \Big\| s \mapsto \int_0^1 \Big| \sum_{(n,I)\in \mathcal{D}_{\omega}} \sum_{K\in \mathcal{C}_I^n} r_K(u)a_I^n h_K(s)  \Big|\, \mathrm{d}u \Big\|_X.
  \end{equation}
  Now observe that for every $s\in [0,1)$ and every pair~$(n,I)\in \mathcal{D}_{\omega}$, there is at most one interval $K\in \mathcal{C}_I^n$ such that $s\in K$. Thus, in~\eqref{eq:17}, we may replace the corresponding Rademacher function~$r_K$ by~$r_I^n$ whenever~$K\in \mathcal{C}_I^n$, which completes the proof.
\end{proof}
\begin{rem}\label{rem:dfn-does-not-depend-on-choice}
  The resulting space in \Cref{dfn:independent-sum} does not depend on the choice of the sets~$\tau_n$:
  Let $(\tilde{h}_I^n)_{(n,I)\in \mathcal{D}_{\omega}}$ be another system of functions in~$H$, originating from a different sequence of sets $(\tilde{\tau}_n)_{n\in \mathbb{N}_0}$ and constructed according to \Cref{dfn:independent-sum}. Then the linear extension of $h_I^n\mapsto \tilde{h}_I^n$ provides an isometric isomorphism between the subspaces $[h_I^n]_{(n,I)\in \mathcal{D}_{\omega}}$ and $[\tilde{h}_I^n]_{(n,I)\in \mathcal{D}_{\omega}}$ of~$X(\mathbf{r})$. To see this, note that for every~$N\in \mathbb{N}$, both $(h_I^n)_{I\in \mathcal{D}_{\le n}}$ ($0\le n\le N$) and $(\tilde{h}_I^n)_{I\in \mathcal{D}_{\le n}}$ ($0\le n\le N$) are independent distributional copies of $(h_I)_{I\in \mathcal{D}_{\le n}}$ ($0\le n\le N$). This implies that the families
  \begin{equation}\label{eq:1}
    (h_I^n : I\in \mathcal{D}_{\le n},\, 0\le n\le N)\qquad  \text{and} \qquad  (\tilde{h}_I^n : I\in \mathcal{D}_{\le n},\, 0\le n\le N)
  \end{equation}
  also have the same distribution, and upon inspecting the formula for~$\|x\|_Y$ in \Cref{lem:independent-sum-norm}, we obtain the claimed result.
  Finally, in passing, we note that all linear combinations $\sum_{n=0}^N\sum_{I\in \mathcal{D}_{\le n}} a_I^nh_I^n$ and $\sum_{n=0}^N\sum_{I\in \mathcal{D}_{\le n}} a_I^n \tilde{h}_I^n$ have the same distribution as the function
  \begin{equation*}
    [0,1)^N \to \mathbb{R},\qquad
    (s_0,\dots,s_N) \mapsto \sum_{n=0}^N \sum_{I\in \mathcal{D}_{\le n}} a_I^n h_I(s_n).
  \end{equation*}
\end{rem}

The following result will be widely used throughout this paper.
\begin{lem}\label{lem:schauder-basis}
  Let~$Y = X_0(\mathbf{r})$ be a Haar system Hardy space. Then~$(h_I^n)_{(n,I)\in \mathcal{D}_{\omega}}$, in the order defined in \Cref{sec:notation}, is a monotone Schauder basis of~$Y_{\omega}$, and it is $1$-unconditional if~$\mathbf{r}$ is independent. Moreover, put $H_\omega^n = \operatorname{span}\{ h_I^n \}_{I\in \mathcal{D}_{\le n}}$ for $n\in \mathbb{N}_0$. Then $(H_{\omega}^n)_{n=0}^{\infty}$ is a $1$-suppression unconditional finite-dimensional Schauder decomposition of~$Y_{\omega}$, and for every~$n\in \mathbb{N}_0$, the subspace $H_{\omega}^n\subset Y_{\omega}$ is isometrically isomorphic to~$Y_n\subset Y$.
\end{lem}
\begin{proof}
  If~$\mathbf{r}$ is independent, then the $1$-unconditionality is immediate from \Cref{lem:independent-sum-norm}.
  Otherwise, we may assume that the sets $\tau_n$ in \Cref{dfn:independent-sum} satisfy $\max \tau_{n-1} < \min \tau_n$ for all~$n\in \mathbb{N}$ (see \Cref{rem:dfn-does-not-depend-on-choice}). Then the proof that~$(h_I^n)_{(n,I)\in \mathcal{D}_{\omega}}$ is a monotone Schauder basis is analogous to \cite[Proposition~2.13~(c)]{MR4430957}, exploiting that the functions~$(h_I^n)_{(n,I)\in \mathcal{D}_{\omega}}$ are blocks of the standard Haar system whose frequencies increase with~$n$.

  To prove the second assertion, let $\sigma\subset \mathbb{N}_0$ be finite, and let $P_\sigma\colon Y_{\omega}\to Y_{\omega}$ be the projection given by $P_\sigma h_I^n = h_I^n$ for $n\in \sigma$ and $P_\sigma h_I^n = 0$ for $n\notin \sigma$. Again exploiting \Cref{rem:dfn-does-not-depend-on-choice}, we may assume that the sets $\tau_n$ are arranged in such a way that for some fixed $N\in \mathbb{N}_0$, we have $\max \tau_n \le  N < \min \tau_k$ whenever $n\in \sigma$ and $k\in \mathbb{N}_0\setminus \sigma$. Then $P_\sigma$ has norm one because the \emph{standard} Haar system $(h_I)_{I\in \mathcal{D}}$ is a monotone Schauder basis of~$Y$ and $P_\sigma$ is just the associated basis projection onto~$\operatorname{span}\{ h_I \}_{I\in \mathcal{D}_{\le N}}$, restricted to~$Y_{\omega}$.  Alternatively, the second assertion can be proved via conditional expectations, utilizing \Cref{lem:cond-expect-sum} together with \Cref{lem:cond-expect}.

  To see the final statement, note that by \Cref{rem:dfn-does-not-depend-on-choice}, for fixed $n\in \mathbb{N}_0$, we may assume that $\tau_n < \tau_0 < \tau_1 < \dots$ as well as $\tau_n = \{ 0,1,\dots,n \}$ and $h_I^n = h_I$ for all $I\in \mathcal{D}_{\le n}$.
\end{proof}
\begin{rem}
  For $(n,I)\in \mathcal{D}_{\omega}$, the coordinate functional associated with the basis element~$h_I^n$ is given by $x\mapsto |I|^{-1}\int h_I^n x$, $x\in H_{\omega}$. In particular, $H_{\omega} = \operatorname{span}\{ h_I^n \}_{(n,I)\in \mathcal{D}_{\omega}}$ can be identified with a subspace of~$Y_{\omega}^{*}$.
\end{rem}

\begin{rem}
  In general, the Schauder decomposition in \Cref{lem:schauder-basis} is not $1$-unconditional: For example, $f = h_{[0,1)}^1 + 2h_{[0,\frac{1}{2})}^1$ and $g = h_{[0,1)}^2 + 2h_{[0,\frac{1}{2})}^2$ satisfy $\|f + g\|_{L^1} \ne \|f - g\|_{L^1}$.
\end{rem}

\begin{rem}\label{rem:isometric-isomorphism-Jn}
  Let~$Y$ be a Haar system Hardy space, and for~$n\in \mathbb{N}_0$, let $J_n\colon Y_n\to H_{\omega}^n$ be the isometric isomorphism given by~$J_nh_I = h_I^n$, $I\in \mathcal{D}_{\le n}$, and let $P_n\colon Y_{\omega}\to Y_{\omega}$ be the norm-one projection onto~$H_{\omega}^n$ defined by
  \begin{equation*}
    P_n\Bigl( \sum_{m=0}^{\infty}\sum_{I\in \mathcal{D}_{\le m}} a_I^m h_I^m  \Bigr)
    = \sum_{I\in \mathcal{D}_{\le n}} a_I^n h_I^n
  \end{equation*}
  (we know from \Cref{lem:schauder-basis} that $P_n$ has norm one).
 First, note that~$J_n$ preserves the tree structure of the Haar system: If $I,J\in \mathcal{D}_{\le n}$, then $I\subset J^{\pm }$ if and only if $\operatorname{supp} h_I^n\subset [h_J^n = \pm 1]$, and we have $I\cap J=\emptyset$ if and only if $h_I^n$ and~$h_J^n$ are disjointly supported. Moreover, any function $x\in Y_n$ has the same distribution as~$J_nx\in H_{\omega}^n$ and, in particular, $|\operatorname{supp} x| = |\operatorname{supp} J_n x|$.
  Now let $T\colon Y_{\omega}\to Y_{\omega}$ be a bounded linear operator. Then the operator~$T_n \colon Y_n\to Y_n$ defined by
  \begin{equation*}
    T_n = J_n^{-1}P_nTJ_n
  \end{equation*}
  satisfies $\|T_n\|\le \|T\|$, and with respect to the Haar basis, it has the same entries as~$T$ restricted to $H_{\omega}^n$:
  \begin{equation}\label{eq:22}
    \langle h_I, T_n h_J \rangle = \langle h_I^n, T h_J^n \rangle, \qquad I,J\in \mathcal{D}_{\le n}.
  \end{equation}
  In particular, if~$T$ has $\delta$-large (positive) diagonal, then so does~$T_n$.
  More generally, for all $x,y\in Y_n$, we have $\langle x,y \rangle = \langle J_nx, J_n y \rangle$ and thus $\langle x, T_n y \rangle = \langle J_nx, T J_ny \rangle$.
  Finally, \eqref{eq:22} together with \Cref{pro:HSHS-1d}~\eqref{pro:HSHS-1d:iii} implies that the diagonal entries of~$T$ lie between~$-\|T\|$ and~$\|T\|$:
  \begin{equation*}
    \frac{|\langle h_I^n, T h_I^n \rangle|}{|I|}  = \frac{|\langle h_I, T_n h_I \rangle|}{|I|}\le \frac{\|h_I\|_{Y^{*}}\|h_I\|_Y}{|I|} \|T\| = \|T\|.
  \end{equation*}
\end{rem}

The following standard result will be used in the heart of the proof of our diagonalization result, \Cref{pro:diagonalization}.

\begin{lem}\label{lem:haar-weakly-null}
  Let~$Y\in \mathcal{HH}_0(\delta)$ be a Haar system Hardy space, and assume that the sequence of standard Rademacher functions~$(r_n)_{n=0}^{\infty}$ is weakly null in~$Y$. For~$n\in \mathbb{N}_0$, let $\mathcal{K}_n$ denote the set of all functions~$f\in H_0$ of the form $f = \sum_{K\in \mathcal{B}} \theta_K h_K$, where $\mathcal{B}\subset \mathcal{D}_n$ and $(\theta_K)_{K\in \mathcal{D}_n}\in \{ \pm 1 \}^{\mathcal{D}_n}$. Then for every $y\in Y$ and $y^{*}\in Y^{*}$, we have
  \begin{equation*}
    \lim_{n\to \infty} \sup_{f\in \mathcal{K}_n} |\langle f, y \rangle| = 0
    \qquad \text{and}\qquad
    \lim_{n\to \infty} \sup_{f\in \mathcal{K}_n} |\langle y^{*}, f \rangle| = 0.
  \end{equation*}
  In particular, every subsequence of the basis~$(h_I^n)_{(n,I)\in \mathcal{D}_{\omega}}$ constructed in \Cref{dfn:independent-sum} is weakly null in~$Y$ (and hence in~$Y_{\omega}$) as well as weak* null in~$Y^{*}$ (and hence in~$Y_{\omega}^{*}$).
\end{lem}
\begin{proof}
  In \cite[Lemma~4.3]{MR4430957}, this is proved for Haar system spaces. The same proof also works in a Haar system Hardy space~$Y$.
\end{proof}

\section{Embeddings and projections on $Y_{\omega}$}\label{sec:embeddings-projections}

In this section, we lay the foundation for proving our main results. We introduce the notions of \emph{faithful} and \emph{almost faithful Haar systems in~$Y_{\omega}$}, and we prove that the canonical operators $A$ and~$B$ associated with such a system are bounded (\Cref{thm:operators-A-B} and \Cref{thm:operators-A-B-almost-faithful}). This constitutes one of the main parts of the paper and will be used throughout the subsequent sections: In every reduction step, we will first construct an (almost) faithful Haar system in~$Y_{\omega}$ with certain additional properties and then use the associated operators $A$ and~$B$ to construct a factorization, eventually reducing a given operator~$T$ to a scalar operator.

Note that unlike $\ell^p$- and $c_0$-direct sums, it is not clear whether a sequence of uniformly bounded linear operators acting on the components~$H_\omega^n$ induces a bounded linear operator on~$Y_\omega$. Hence, the analogous results in~\cite{MR4894823,MR4884827} for the operators $A,B$ associated with a faithful Haar system in~$Y$ or~$Y_n$ do not transfer directly to our setting. In the case where the underlying space is the dyadic Hardy space~$H^1$, the work of M.~Rzeszut and M.~Wojciechowski~\cite[Section~5]{MR3648869} contains sufficient conditions involving \emph{$K$-closedness} which ensure that the resulting operator on the independent sum of the spaces~$H^1_n$ is indeed bounded. Here, instead, we pursue a more direct approach, which is adapted to our operators and exploits that~$BA$ acts like a conditional expectation. We start by proving some basic results on the boundedness of diagonal operators and conditional expectations on~$Y_{\omega}$.
\begin{lem}\label{lem:diagonal-operator-bounded}
  Let~$Y$ be a Haar system Hardy space, let~$\eta > 0$, and let~$(d_I^n)_{(n,I)\in \mathcal{D}_{\omega}}$ be a sequence of scalars such that~$|d_I^n|\le 8^{-1}4^{-n}\eta$ for all~$(n,I)\in \mathcal{D}_{\omega}$. Then the diagonal operator~$D\colon Y_{\omega}\to Y_{\omega}$ with~$D h_I^n = d_I^nh_I^n$, $(n,I)\in \mathcal{D}_{\omega}$, is bounded and satisfies~$\|D\|\le \eta$.
\end{lem}
\begin{proof}
  Let $x = \sum_{(n,I)\in \mathcal{D}_{\omega}} a_I^n h_I^n\in H_{\omega}$. Then, using that~$(h_I^n)_{(n,I)\in \mathcal{D}_{\omega}}$ is a Schauder basis of~$Y_{\omega}$, we have $\|a_I^nh_I^n\|_Y\le 2 \|x\|_Y$ for all $(n,I)\in \mathcal{D}_{\omega}$.
  Thus, it follows that
  \begin{align*}
    \|Dx\|_Y &\le \sum_{n=0}^{\infty} \sum_{I\in \mathcal{D}_{\le n}} |d_I^n| \|a_I^nh_I^n\|_Y
              \le 2\|x\|_Y\sum_{n=0}^{\infty} 2^{n+1}\cdot 8^{-1}4^{-n}\eta = \eta \|x\|_Y.\qedhere
  \end{align*}
\end{proof}

Next, we prove that if~$\mathcal{F}$ is a $\sigma$-algebra generated by sub-$\sigma$-algebras of the independent ``components'' of the space~$Y_{\omega}$, then the conditional expectation with respect to~$\mathcal{F}$, restricted to~$Y_\omega$, can be written as the sum of the individual conditional expectations. This result is standard, yet it will be crucial for proving the main results of this section. For convenience, we provide a full proof.
\begin{lem}\label{lem:cond-expect-sum}
  Let~$Y$ be a Haar system Hardy space, and let $Y_{\omega}$ and its basis $(h_I^n)_{(n,I)\in \mathcal{D}_{\omega}}$ be constructed as in \Cref{dfn:independent-sum}. Let $m\in \mathbb{N}_0$, and for every $0\le n\le m$, let $\mathcal{F}_n$ be a sub-$\sigma$-algebra of $\mathcal{H}_n := \sigma(\{ h_I^n\}_{I\in \mathcal{D}_{\le n}})$. Put $\mathcal{F} = \sigma(\mathcal{F}_0\cup \dots\cup \mathcal{F}_m)$. Then for every $x\in H_{\omega}$, we have
  \begin{equation}\label{eq:7}
    \mathbb{E}^{\mathcal{F}}x = \sum_{n=0}^m \mathbb{E}^{\mathcal{F}_n}x.
  \end{equation}
\end{lem}
\begin{proof}
  It suffices to prove the equation for~$x = h_I^k$, where $(k,I)\in \mathcal{D}_{\omega}$.
  Recall that $(\mathcal{H}_n)_{n\in \mathbb{N}_0}$ is an independent family of $\sigma$-algebras (see \Cref{rem:faithful-from-rademachers}), and we have $\mathcal{F}_n\subset \mathcal{H}_n$ for all $0\le n\le m$. Hence, if~$k > m$, then both sides of~\eqref{eq:7} vanish. Now assume that $0\le k\le m$. Since $\mathbb{E}^{\mathcal{F}_n} h_I^k = 0$ for all $n\ne k$, we have to show that
  \begin{equation*}
    \mathbb{E}^{\mathcal{F}} h_I^k = \mathbb{E}^{\mathcal{F}_k} h_I^k.
  \end{equation*}
  We can write $h_I^k = \chi_B - \chi_C$ for some sets $B,C\in \mathcal{H}_k$,
  so it suffices to prove that
  \begin{equation*}
    \mathbb{E}^{\mathcal{F}} \chi_B = \mathbb{E}^{\mathcal{F}_k} \chi_B,\qquad B\in \mathcal{H}_k.
  \end{equation*}
  To this end, let $B\in \mathcal{H}_k$, and let~$A$ be an atom of the $\sigma$-algebra~$\mathcal{F}$, i.e., $A = A_0\cap \dots\cap A_m$, where~$A_n$ is an atom of~$\mathcal{F}_n$ for every $0\le n\le m$. Then, by independence, we obtain
  \begin{equation}\label{eq:18}
    \int_A \chi_B = |A_0\cap \dots\cap A_m\cap B|
                  =  |A_k \cap B|\cdot \prod_{\substack{n=0\\ n\ne k}}^m|A_n|.
  \end{equation}
  Moreover, we have
  \begin{equation*}
    \int_A \mathbb{E}^{\mathcal{F}_k} \chi_B
    = \int_A \sum_{L\text{ atom of }\mathcal{F}_k} \biggl( \frac{1}{|L|} \int_L \chi_B \biggr) \chi_L
    = \sum_{L\text{ atom of }\mathcal{F}_k} \frac{|L\cap B|}{|L|} \int_A \chi_L.
  \end{equation*}
  If $L$ is an atom of~$\mathcal{F}_k$, then since the same is true for $A_k$, it follows that
  \begin{equation*}
    \int_A \chi_L = |A_1\cap \dots\cap (A_k\cap L)\cap \dots\cap A_m|
    =
    \begin{cases}
      0, & L\ne A_k,\\
      |A|, & L = A_k.
    \end{cases}
  \end{equation*}
  Hence,
  \begin{equation*}
    \int_A \mathbb{E}^{\mathcal{F}_k}\chi_B = \frac{|A_k \cap B|}{|A_k|}\cdot \prod_{n=0}^m |A_n|,
  \end{equation*}
  and this is equal to the right-hand side of~\eqref{eq:18}.
\end{proof}

For a first application of \Cref{lem:cond-expect-sum}, we return to diagonal operators. We provide a sufficient condition for the boundedness of certain diagonal operators with entries in~$\{ 0,1 \}$ which act on $Y_{\omega}$ similarly to dyadic stopping times.
\begin{lem}\label{lem:multiplier-zero-one}
  Let $Y = X_0(\mathbf{r})$ be a Haar system Hardy space. Moreover, let $(d_I^n)_{(n,I) \in \mathcal{D}_{\omega}}$ be a sequence in $\lbrace 0,1 \rbrace$ such that for every $n \in \mathbb{N}$ and $I\in \mathcal{D}_{\le n-1}$, if $d_I^n = 0$, then $d_{I^+}^n = d_{I^-}^n =0$. Then $Dh_I^n = d_I^n h_I^n$, $(n,I)\in \mathcal{D}_{\omega}$, extends to a bounded linear operator $D \colon Y_{\omega} \to Y_{\omega}$ with $\|D\| \leq 1$.
\end{lem}
\begin{proof}
  If $\mathbf{r}$ is independent, then the 1-unconditionality of the basis~$(h_I^n)_{(n,I)\in \mathcal{D}_{\omega}}$ of~$Y_{\omega}$ implies the result. Thus, it is sufficient to prove the result for the case where $\mathbf{r}$ is constant. To this end, let $x = \sum_{n=0}^{N} \sum_{I \in \mathcal{D}_{\leq n}} a_{I}^{n} h_{I}^{n}\in H_{\omega}$ for some $N\in \mathbb{N}$. We will construct an appropriate $\sigma$-algebra $\mathcal{F}$ such that $Dx = \mathbb{E}^{\mathcal{F}}x$. For every $0\le n\le N$ and $I\in \mathcal{D}_n$, let $\Gamma_n(I^{\pm }) = [h_I^n = \pm 1]$. Moreover, for $0\le n\le N$ and $J\in \mathcal{D}_{n+1}$, put $M_n(J)= \Gamma_n(J) \cup \bigcup \{ \supp h_I^n : J\subset I,\, d_I^n = 0 \}$ and observe that $\mathcal{M}_n = \{ M_n(J) : J\in \mathcal{D}_{n+1} \}$ is a partition of $[0,1)$. Next, for every $0\le n\le N$, let $\mathcal{F}_n$ be the $\sigma$-algebra generated by $\mathcal{M}_N$, and observe that $\mathcal{F}_n \subset \mathcal{H}_n$.
  Finally, put $\mathcal{F} = \sigma(\mathcal{F}_0\cup \dots\cup \mathcal{F}_N)$. Then \Cref{lem:cond-expect-sum} implies that $\mathbb{E}^{\mathcal{F}}x = \sum_{n=0}^N \sum_{I \in \mathcal{D}_{\leq n}} a_I^n \mathbb{E}^{\mathcal{F}_{n}} h_I^n$. For every  $0 \leq n \leq N $ and $I \in \mathcal{D}_{\leq n}$, the same arguments as in the proof of \cite[Lemma~2.3.8]{speckhofer:2025} show that that if $d_{I}^{n} = 0$, then $\mathbb{E}^{\mathcal{F}_{n}}h_{I}^{n} = 0$, and if  $d_{I}^{n} = 1$, then $\mathbb{E}^{\mathcal{F}_{n}}h_{I}^{n} = h_{I}^{n}$ (since $h_{I}^{n}$ is measurable with respect to $\mathcal{F}_{n}$). Thus, $Dx = \mathbb{E}^{\mathcal{F}}x$, and the conclusion follows from \Cref{lem:cond-expect}.
\end{proof}

Next, we introduce \emph{faithful} and \emph{almost faithful Haar systems in~$Y_{\omega}$}. They are constructed from a sequence of \emph{finite} (almost) faithful Haar systems in~$Y$, which are combined to an infinite system in~$Y_{\omega}$ using the isometries~$J_n\colon Y_n \to H_{\omega}^n$, $n\in \mathbb{N}_0$, from \Cref{rem:isometric-isomorphism-Jn}.

\begin{dfn}\label{dfn:faithful-system-independent-sum}
Let $N\colon \mathbb{N}_0\to \mathbb{N}_0$ be a strictly increasing function, and for every $n\in \mathbb{N}_0$, let $(\tilde{h}_I^n)_{I\in \mathcal{D}_{\le n}}$ be a finite almost faithful Haar system in $Y_{N(n)}$ (note that $N(n)\ge n$). For every $n\in \mathbb{N}_0$, we can apply the isometric isomorphism $J_{N(n)}$ defined in \Cref{rem:isometric-isomorphism-Jn} to obtain functions $\tilde{b}_I^n = J_{N(n)} \tilde{h}_I^n \in  H_{\omega}$, $I\in \mathcal{D}_{\le n}$. We call the resulting system $(\tilde{b}_I^n)_{(n,I)\in \mathcal{D}_{\omega}}$ an \emph{almost faithful Haar system in~$Y_{\omega}$}. If all systems $(\tilde{h}_I^n)_{I\in \mathcal{D}_{\le n}}$ are faithful, we call $(\tilde{b}_I^n)_{(n,I)\in \mathcal{D}_{\omega}}$ a \emph{faithful Haar system in $Y_{\omega}$}, and we will usually denote it as $(b_I^n)_{(n,I)\in \mathcal{D}_{\omega}}$.
\end{dfn}
\begin{rem}
  Suppose that for every $n\in \mathbb{N}_0$, the system $(\tilde{h}_I^n)_{I\in \mathcal{D}_{\le n}}$ is defined as
  \begin{equation*}
    \tilde{h}_I^n = \sum_{K\in \mathcal{B}_I^n} \theta_K^{N(n)} h_K, \qquad I\in \mathcal{D}_{\le n}
  \end{equation*}
  where $(\mathcal{B}_I^n)_{I\in \mathcal{D}_{\le n}}$ is a family of finite subcollections of $\mathcal{D}_{\le N(n)}$ and $\theta_K^m\in \{ \pm 1 \}$ for all $(m,K)\in \mathcal{D}_\omega$. Then the resulting system $(\tilde{b}_I^n)_{(n,I)\in \mathcal{D}_{\omega}}$ is given by
  \begin{equation*}
    \tilde{b}_I^n = \sum_{K\in \mathcal{B}_I^n} \theta_K^{N(n)} h_K^{N(n)},\qquad (n,I)\in \mathcal{D}_{\omega}.
  \end{equation*}
\end{rem}

\subsection{Operators associated with a faithful Haar system}

In the following theorem we define the fundamental operators $\hat{A}$ and~$\hat{B}$ associated with a faithful Haar system in~$Y_{\omega}$, constructed as in \Cref{dfn:faithful-system-independent-sum}, and we prove that they are bounded. Moreover, we will see that the operators satisfy $\hat{A}\hat{B} = I_{Y_{\omega}}$, which implies that $\hat{B} \hat{A}$ is a projection onto the closed linear span of the functions $b_I^n = \hat{B} h_I^n$, $(n,I)\in \mathcal{D}_{\omega}$. Our proof is based on \cite[Proposition~7.1]{MR4894823}, which we transfer to the setting of independent sums, utilizing \Cref{lem:cond-expect-sum}, our result on sums of conditional expectations.

\begin{thm}\label{thm:operators-A-B}
  Let $Y$ be a Haar system Hardy space, and let $(b_I^n)_{(n,I) \in \mathcal{D}_{\omega}}$ be a faithful Haar system in~$Y_{\omega}$ (see~\Cref{dfn:faithful-system-independent-sum}). Define the operators $\hat{A},\hat{B} \colon Y_{\omega}\to  Y_{\omega}$ by
\begin{equation}\label{eq:15}
  \hat{B}x =  \sum_{(n,I) \in \mathcal{D}_{\omega}} \frac{\langle h_I^n, x \rangle}{\vert I \vert} b_I^n
  \qquad \text{and} \qquad
  \hat{A}x =  \sum_{(n,I) \in \mathcal{D}_{\omega}}   \frac{\langle b_I^n, x \rangle}{\vert I \vert} h_I^n.
\end{equation}
Then we have $\hat{A}\hat{B} = I_{Y_{\omega}}$ and $\| \hat{A} \| =  \| \hat{B} \| =1$.
\end{thm}
\begin{proof}
Write $Y = X_0(\mathbf{r})$, where~$X$ is a Haar system space and~$\mathbf{r}$ is either constant or independent.
Let the faithful Haar system $(b_I^n)_{(n,I) \in \mathcal{D}_{\omega}}$ be given by
\begin{equation*}
  b_I^n = \sum_{K \in \mathcal{B}_{I}^{n}} \theta_K^{N(n)} h_K^{N(n)},
  \qquad (n,I)\in \mathcal{D}_{\omega},
\end{equation*}
where $N\colon \mathbb{N}_0\to \mathbb{N}_0$ is a strictly increasing function, $\mathcal{B}_I^n\subset \mathcal{D}_{\le N(n)}$ for all $(n,I)\in \mathcal{D}_{\omega}$, and $(\theta_K^m)_{(m,K)\in \mathcal{D}_{\omega}}\in \{ \pm 1 \}^{\mathcal{D}_{\omega}}$.

First, we prove that the operator $\hat{B}$ is an isometry. To this end, let $x =  \sum_{ (n,I) \in \mathcal{D}_{\omega} } a_I^nh_I^{n} \in  H_{\omega}$. Then we have
\begin{equation*}
  \hat{B}x =  \sum_{(n,I) \in \mathcal{D}_{\omega}}  a_I^n b_I^n.
\end{equation*}
Using \Cref{lem:independent-sum-norm}, we obtain
\begin{equation}\label{eq:3}
\| \hat{B}x \|_{X(r)} = \Big\| s \mapsto \int_0^1 \Big|  \sum_{(n,I) \in \mathcal{D}_{\omega}} a_I^n  \sum_{K \in \mathcal{B}_I^n} \theta_K^{N(n)} r_K^{N(n)}(u) h_K^{N(n)}(s) \Big| du \Big\|_X.
\end{equation}
Now, for fixed $s\in [0,1)$ and $(n,I)\in \mathcal{D}_{\omega}$, there is at most one interval $K\in \mathcal{B}_I^n$ such that $s\in \operatorname{supp} (h_K^{N(n)})$. Hence, we may replace $r_K^{N(n)}(u)$ with $r_I^n(u)$ in~\eqref{eq:3}, which yields
\begin{align*}
  \| \hat{B}x \|_{X(r)}
  &= \Big\| s \mapsto \int_0^1 \Big|  \sum_{(n,I) \in \mathcal{D}_{\omega} }  a_I^n r_I^n(u) \sum_{K \in \mathcal{B}_I^n} \theta_K^{N(n)} h_K^{N(n)}(s) \Big| du \Big\|_X\\
  &=  \Big\| s \mapsto \int_0^1 \Big|  \sum_{(n,I) \in \mathcal{D}_{\omega} } a_I^n r_I^n(u) b_I^n(s) \Big| du \Big\|_X\\
  &= \Big\| s \mapsto \int_0^1 \Big|   \sum_{(n,I) \in \mathcal{D}_{\omega} }  a_I^n r_I^n(u) h_I^n(s) \Big| du \Big\|_X = \| x \|_{X(r)},
\end{align*}
where in the last step, we exploited that the systems $(b_I^n)_{(n,I)\in \mathcal{D}_{\omega}}$ and $(h_I^n)_{(n,I)\in \mathcal{D}_{\omega}}$ have the same joint distribution.

Next, we prove that $\hat{A}$ is bounded. Again, we fix $x = \sum_{ (n,I) \in \mathcal{D}_{\omega} }  a_I^nh_I^n \in H_{\omega}$. Let $m \in \mathbb{N}_0$ be sufficiently large such that $a_K^{N(n)}=0$ for every $n > m$ and $K \in \mathcal{D}_{\le N(n)}$ (hence, $\langle b_I^n, x \rangle = 0$ for all $n > m$ and~$I\in \mathcal{D}_{\le n}$). For every $n \in \{ 0,\dots,m \}$, put
\begin{equation*}
  \mathcal{F}_n = \sigma (\{b_I^n\}_{I \in \mathcal{D}_n}) \subset \mathcal{H}_{N(n)},
\end{equation*}
where $\mathcal{H}_{N(n)}$ is the $\sigma$-algebra generated by $\{h_I^{N(n)}\}_{I\in \mathcal{D}_{\le N(n)}}$.
Now fix $n\in \{ 0,\dots,m \}$. Since $(\mathcal{H}_l)_{l=0}^\infty$ is an independent family of $\sigma$-algebras, it follows that for every $(l,K)\in \mathcal{D}_{\omega}$ with $l\ne N(n)$, we have
\begin{equation}\label{eq:20}
  \mathbb{E}^{\mathcal{F}_n} h_K^l = 0.
\end{equation}
On the other hand, for $K\in \mathcal{D}_{\le N(n)}$, we claim that
\begin{equation}\label{eq:21}
 \mathbb{E}^{\mathcal{F}_n} h_K^{N(n)} =
    \begin{cases}
      0, &  K \notin \bigcup_{I \in \mathcal{D}_{\leq n}} \mathcal{B}_I^n \\
      \theta_K^{N(n)} \frac{|K|}{|I|}b_I^n, & K \in \mathcal{B}_I^n, I \in \mathcal{D}_{\leq n }.
    \end{cases}
\end{equation}
We prove the above claim. For every $J \in \mathcal{D}_n$, let $\Gamma_{J^{\pm }} = [b_J^n = \pm 1]$. We have the expansion
\begin{equation}\label{eq:23}
  \chi_{\Gamma_I} = |I| \chi_{[0,1)} + \sum_{\substack{J \in \mathcal{D}_{\leq n}\\ J \supset I}} \frac{|I|}{|J|} h_J(I) b_J^n,\qquad I\in \mathcal{D}_{n+1},
\end{equation}
where $h_J(I)$ is the value that~$h_J$ assumes on~$I$. This is analogous to the corresponding expansion of~$\chi_I$ in the original Haar system (see also~\cite[Proposition~3.5.1.]{speckhofer:2025}).
Now let $K \in \mathcal{B}_{J_0}^n$, where $J_0 \in \mathcal{D}_{\leq n}$. For every $I \in \mathcal{D}_{n+1}$ with $I \subset J_0$, we have
\begin{equation*}
\langle \chi_{\Gamma_{I}}, h_K^{N(n)} \rangle = \sum_{\substack{J \in \mathcal{D}_{\leq n}\\ J \supset I}} \frac{|I|}{|J|} h_J(I) \langle b_J^n, h_K^{N(n)} \rangle = \frac{\vert I \vert}{\vert J_0 \vert} h_{J_0}(I) \theta_K^{N(n)} \vert K \vert.
\end{equation*}
Moreover, if $I\in \mathcal{D}_{n+1}$ and $I\not\subset J_0$, then the functions~$\chi_{\Gamma_I}$ and~$h_K^{N(n)}$ are disjointly supported. This yields
\begin{equation*}
\mathbb{E}^{\mathcal{F}_n} h_K^{N(n)} = \sum_{\substack{I \in \mathcal{D}_{n+1}\\ I \subset J_0}} \frac{1}{|I|} \langle \chi_{\Gamma_I}, h_K^{N(n)} \rangle \chi_{\Gamma_I} = \theta_K^{N(n)}\frac{|K|}{|J_0|} \sum_{\substack{I\in \mathcal{D}_{n+1}\\ I\subset J_0}} h_{J_0}(I)\chi_{\Gamma_I} = \theta_K^{N(n)} \frac{\vert K \vert}{|J_0|} b_{J_0}^n.
\end{equation*}
Finally, if $K \notin \bigcup_{I \in \mathcal{D}_{\leq n}} \mathcal{B}_I^n$, then we have $\mathbb{E}^{\mathcal{F}_n} h_K^{N(n)} = 0$ since, by~\eqref{eq:23}, $\langle \chi_{\Gamma_I}, h_K^{N(n)} \rangle = 0$ for all~$I\in \mathcal{D}_{n+1}$. This completes the proof of~\eqref{eq:21}.

Next, we define
\begin{equation*}
  \mathcal{C} = \mathcal{D}_{\omega}\setminus \bigl\{ (N(n),K) : 0\le n\le m,\ I\in \mathcal{D}_{\le n},\ K\in \mathcal{B}_I^n  \bigr\}
\end{equation*}
and split the function~$x$ into two parts accordingly (if~$\mathbf{r}$ is constant, replace $r_K^{N(n)}(u)$ and~$r_K^n(u)$ by~$1$ in the following computations). This yields
\begin{align*}
  \| x \|_{X(r)} = \Big\| s \mapsto \int_0^1 \Big|  \sum_{n=0}^m \sum_{I \in \mathcal{D}_{\leq n}}  \sum_{K \in \mathcal{B}_I^n} a_K^{N(n)}  r_K^{N(n)}(u) h_K^{N(n)}(s) + \sum_{(l,K) \in \mathcal{C}} a_K^l r_K^l(u) h_K^l(s) \Big|\, \mathrm{d}u \Big\|_X&\\
                = \Big\| s \mapsto \int_0^1 \int_{0}^{1} \Big|  \sum_{n=0}^m \sum_{I \in \mathcal{D}_{\leq n}} r_I^n(u) \sum_{K \in \mathcal{B}_I^n} a_K^{N(n)} h_K^{N(n)}(s) + \sum_{(l,K) \in \mathcal{C}} a_K^l r_K^l(v) h_K^l(s) \Big|\, \mathrm{d}u\, \mathrm{d}v \Big\|_X,&
\end{align*}
where we again replaced $r_K^{N(n)}(u)$ with $r_I^n(u)$ in the first sum.
Define $\mathcal{F} = \sigma(\mathcal{F}_1\cup \dots\cup \mathcal{F}_m)$. Using Jensen's inequality for conditional expectations, \Cref{lem:cond-expect} and \Cref{pro:HS-1d}~\eqref{pro:HS-1d:v}, we obtain
\begin{align*}
  \| x \|_{X(r)} \geq \Big\| s \mapsto \int_0^1 \int_0^1 \Big| &\sum_{n=0}^m \sum_{I \in \mathcal{D}_{\leq n}} r_I^n(u) \sum_{K \in \mathcal{B}_I^n} a_K^{N(n)} (\mathbb{E}^{\mathcal{F}} h_{K}^{N(n)})(s)\\
 &\qquad + \sum_{(l,K) \in \mathcal{C}} a_K^l r_K^l(v) (\mathbb{E}^{\mathcal{F}} h_K^l)(s) \Big|\, \mathrm{d}u\, \mathrm{d}v \Big\|_{X}.
\end{align*}
By \Cref{lem:cond-expect-sum} and \eqref{eq:20},~\eqref{eq:21}, this implies that
\begin{equation*}
  \| x \|_{X(r)} \geq \Big\| s \mapsto \int_0^1 \Big|  \sum_{n=0}^m \sum_{I \in \mathcal{D}_{\leq n}} r_I^n(u) \sum_{K \in \mathcal{B}_I^n} a_K^{N(n)} (\mathbb{E}^{\mathcal{F}_n} h_K^{N(n)})(s) \Big|\, \mathrm{d}u  \Big\|_X.
\end{equation*}
Again, from~\eqref{eq:21} we obtain that
\begin{align*}
  \| x \|_{X(r)} &\geq  \Big\| s \mapsto \int_0^1 \Big|  \sum_{n=0}^m \sum_{I \in \mathcal{D}_{\leq n}} r_I^n(u)  \sum_{K \in \mathcal{B}_I^n} a_K^{N(n)} \theta_K^{N(n)} \frac{|K|}{|I|}\,  b_I^n(s) \Big|\, \mathrm{d}u \Big\|_X\\
  &= \Big\| s \mapsto \int_0^1 \Big|  \sum_{n=0}^m \sum_{I \in \mathcal{D}_{\leq n}} r_I^n(u) \frac{\langle b_I^n,x  \rangle}{|I|}  \sum_{K \in \mathcal{B}_I^n} \theta_K^{N(n)} h_K^{N(n)}(s) \Big|\, \mathrm{d}u \Big\|_X
\end{align*}
Changing~$r_I^n(u)$ back to~$r_K^{N(n)}(u)$, we get
\begin{align*}
  \| x \|_{X(r)} &\geq  \Big\| s \mapsto \int_0^1 \Big|  \sum_{n=0}^m \sum_{I \in \mathcal{D}_{\leq n}} \frac{\langle b_I^n,x  \rangle}{|I|}  \sum_{K \in \mathcal{B}_I^n}  \theta_K^{N(n)} r_K^{N(n)}(u)  h_K^{N(n)}(s) \Big| du \Big\|_X\\
                &= \| \hat{B}\hat{A}x \|_{X(r)} = \|  \hat{A}x \|_{X(r)},
\end{align*}
as desired. Thus, $\hat{A},\hat{B}$ extend to bounded linear operators on~$Y_{\omega}$, where both series in~\eqref{eq:15} converge in norm (cf.~\cite[Proposition~3.5.1]{speckhofer:2025}).
\end{proof}

\subsection{Operators associated with an almost faithful Haar system}

Finally, we prove that the operators $A$ and~$B$ associated with an almost faithful Haar system in~$Y_{\omega}$ are bounded (\Cref{thm:operators-A-B-almost-faithful}). To this end, we use the following result which states that any finite almost faithful Haar system in~$Y$ can be extended to a faithful Haar system in~$Y$.

\begin{lem}\label{lem:almost-faithful-to-faithful}
  Let $n,N\in \mathbb{N}_0$, and let $(\tilde{h}_I)_{I\in \mathcal{D}_{\le n}}$ be a finite almost faithful Haar system in~$Y_N$. Then there exists a finite faithful Haar system~$(\hat{h}_I)_{I\in \mathcal{D}_{\le n}}$ in~$Y_{N+n}$ and a Haar multiplier $R\colon Y_{N+n}\to Y_{N+n}$, $R h_K = \rho_K h_K$, with $\rho_K\in \{ 0,1 \}$ for all $K\in \mathcal{D}_{\le N+n}$, such that the following hold:
  \begin{enumerate}[(i)]
    \item\label{lem:almost-faithful-to-faithful:i} For every $K\in \mathcal{D}_{< N + n}$, if $\rho_K = 0$, then $\rho_{K^{\pm }} = 0$.
    \item\label{lem:almost-faithful-to-faithful:ii} If $K\notin \mathcal{D}_{\le N}$, then $\rho_K = 0$.
    \item\label{lem:almost-faithful-to-faithful:iii} $R \hat{h}_I = R \tilde{h}_I = \tilde{h}_I$ for all $I\in \mathcal{D}_{\le n}$.
  \end{enumerate}
\end{lem}
\begin{proof}
  This is a finite version of the construction in Lemma~5.5 and Lemma~5.6 of~\cite{MR4894823}. Note that the numbers $n_0,n_1,\dots,n_n$ in the proof of \cite[Lemma~5.5]{MR4894823} can be chosen as $N, N+1,\dots, N+n$. Hence, the resulting system~$(\hat{h}_I)_{I\in \mathcal{D}_{\le n}}$ is indeed contained in the space~$Y_{N+n}$.
\end{proof}

\begin{thm}\label{thm:operators-A-B-almost-faithful} Let $(\tilde{b}_I^n)_{(n,I) \in \mathcal{D}_{\omega}}$ be an almost faithful Haar system in $Y_{\omega}$, and let $\eta>0$. For $(n,I) \in \mathcal{D}_{\omega}$, put $B_I^n = \operatorname{supp} \tilde{b}_I^n$ and $\mu_n = |B_{[0,1)}^n|$.  Assume that $\mu_n\ge 1/2$ and $\bigl| |I|/|B_I^n|  - 1/\mu_n\bigr| \le 8^{-1}4^{-n}\eta$ for all $(n,I)\in \mathcal{D}_{\omega}$. Then the operators $A,B \colon Y_{\omega} \to  Y_{\omega} $ defined by
\begin{align*}
  Bx =  \sum_{(n,I) \in \mathcal{D}_{\omega}} \frac{\langle h_I^n, x \rangle}{|I|} \tilde{b}_I^n
  \qquad \text{and} \qquad
  Ax =  \sum_{(n,I) \in \mathcal{D}_{\omega}} \frac{\langle \tilde{b}_I^n, x \rangle}{|B_I^n|} h_I^n
\end{align*}
satisfy $AB = I_{Y_{\omega}}$ as well as $\| B \| \leq 1$ and $\| A \| \leq 4 +\eta$.
\end{thm}
\begin{proof}
  Let the almost faithful Haar system~$(\tilde{b}_I^n)_{(n,I)\in \mathcal{D}_{\omega}}$ be given by
  \begin{equation*}
    \tilde{b}_I^n = J_{N(n)} \tilde{h}_I^n,\qquad (n,I)\in \mathcal{D}_{\omega},
  \end{equation*}
  where $N\colon \mathbb{N}_0\to \mathbb{N}_o$ is strictly increasing and for every~$n\in \mathbb{N}_0$, $(\tilde{h}_I^n)_{I\in \mathcal{D}_{\le n}}$ is an almost faithful Haar system in~$Y_{N(n)}$. For every~$n\in \mathbb{N}_0$, we can apply \Cref{lem:almost-faithful-to-faithful} to obtain a faithful Haar system~$(\hat{h}_I^n)_{I\in \mathcal{D}_{\le n}}$ in~$Y_{N(n)+n}$ as well as a Haar multiplier~$R_n$ on~$Y_{N(n)+n}$ with entries~$\rho_I^n\in \{ 0,1 \}$, $I\in \mathcal{D}_{\le N(n)+n}$, such that properties~\eqref{lem:almost-faithful-to-faithful:i}--\eqref{lem:almost-faithful-to-faithful:iii} of~\Cref{lem:almost-faithful-to-faithful} are satisfied. In particular, we have~$R_n \hat{h}_I^n = \tilde{h}_I^n$ for all~$(n,I)\in \mathcal{D}_{\omega}$. By combining the entries of the operators~$R_n$ and padding with zeros, we obtain a diagonal operator $R\colon Y_{\omega}\to Y_{\omega}$ which satisfies the assumptions of \Cref{lem:multiplier-zero-one} (hence, $\|R\|\le 1$) and, in addition,
  \begin{equation*}
    R|_{H_{\omega}^{N(n)+n}} = J_{N(n)+n} R_n J_{N(n)+n}^{-1},\qquad  n\in \mathbb{N}_0.
  \end{equation*}
  Moreover, we define the operator~$S\colon Y_{\omega}\to Y_{\omega}$ by $S h_I^n = h_I^{k(n)}$ for all $(n,I)\in \mathcal{D}_{\omega}$, where $k\colon \mathbb{N}_0\to \mathbb{N}_0$ is any strictly increasing function with the property that $k(N(n)) = N(n) + n$ for all~$n\in \mathbb{N}_0$. Note that~$S$ is an isometry. Next, we put
  \begin{equation*}
    b_I^n = J_{N(n) + n} \hat{h}_I^n,\qquad (n,I)\in \mathcal{D}_{\omega}
  \end{equation*}
  to obtain a faithful Haar system~$(b_I^n)_{(n,I)\in \mathcal{D}_{\omega}}$ in~$Y_{\omega}$. We know from \Cref{thm:operators-A-B} that the operators~$\hat{A},\hat{B}$ associated with~$(b_I^n)_{(n,I)\in \mathcal{D}_{\omega}}$ satisfy $AB = I_{Y_{\omega}}$ and $\|\hat{A}\| = \|\hat{B}\| = 1$.

  Finally, we define the diagonal operator $M\colon Y_{\omega}\to Y_{\omega}$ by $M h_I^n = m_I^n h_I^n$ for all $(n,I)\in \mathcal{D}_{\omega}$, where $m_I^n = |I|/|B_I^n|$. Note that $m_{[0,1)}^n = 1/\mu_n$ for every $n\in \mathbb{N}_0$ and, by the hypothesis, $|m_I^n - 1/\mu_n|\le 8^{-1}4^{-n}\eta$ for all~$(n,I)\in \mathcal{D}_{\omega}$. To prove that~$M$ is bounded, we consider the auxiliary operator $\tilde{M}\colon Y_{\omega}\to Y_{\omega}$ given by $\tilde{M}h_I^n = (1/\mu_n) h_I^n$ for all $(n,I)\in \mathcal{D}_{\omega}$. We know from \Cref{lem:schauder-basis} that $(H_{\omega}^n)_{n=0}^{\infty}$ is a $2$-unconditional finite-dimensional Schauder decomposition of~$Y_{\omega}$. Hence, $\|\tilde{M}\|\le 2 \sup_{n\in \mathbb{N}_0} |1/\mu_n|\le 4$. Thus, it follows from \Cref{lem:diagonal-operator-bounded} that $\|M\|\le \|\tilde{M}\| + \|M - \tilde{M}\|\le 4 + \eta$.

  Now observe that
  \begin{equation*}
    R \hat{B} h_I^n = R b_I^n = J_{N(n)+n} R_n \hat{h}_I^n = S \tilde{b}_I^n,\qquad (n,I)\in \mathcal{D}_{\omega}.
  \end{equation*}
  Hence, we can write $B = S^{-1}R \hat{B}$. Moreover, we have $A = M \hat{A}RS$ since $\langle \tilde{b}_I^n, x \rangle = \langle S^{-1} R b_I^n, x \rangle = \langle b_I^n, RSx \rangle$ for all~$x\in H_\omega$ and $(n,I)\in \mathcal{D}_{\omega}$ (note that~$R$ is a diagonal operator). This implies that $\|B\|\le 1$ and $\|A\|\le 4 + \eta$.
\end{proof}

\section{Diagonalization and level-wise stabilization}

In this section, we prove that under weak assumptions on the space $Y$, any bounded linear operator $T\colon Y_{\omega}\to Y_{\omega}$ can be reduced to a diagonal operator: We show that there exists a bounded linear operator $D\colon Y_{\omega}\to Y_{\omega}$, which is diagonal with respect to $(h_I^n)_{(n,I)\in \mathcal{D}_{\omega}}$, and operators $A,B$ such that $\|D - ATB\|$ is arbitrarily small. Moreover, the entries~$(d_I^n)_{(n,I)\in \mathcal{D}}$ of~$D$ are already stabilized along every level, i.e., $d_I^n\approx d_J^n$ whenever $|I|=|J|$. This is achieved by combining the infinite-dimensional methods from~\cite{MR4839586,MR4894823} with the finite-dimensional diagonalization result from~\cite{MR4884827}, which uses probabilistic techniques based on the work of R.~Lechner~\cite{MR3990955}. The idea of applying probabilistic methods (in particular, large deviation estimates) to stabilize entries of diagonal operators goes back to~\cite{MR4430957}.

In the following, we first construct finite faithful Haar systems which diagonalize the operator~$T$ separately in each finite-dimensional component $H_{\omega}^n$, $n\in \mathbb{N}_0$. Then we combine these constructions to obtain a system $(b_I^n)_{(n,I)\in \mathcal{D}_{\omega}}$, ensuring that also the off-diagonal terms of the form $\langle b_I^n, T b_J^m \rangle$ for $n\ne m$, which involve two different components $H_{\omega}^n$ and~$H_{\omega}^m$, are small. Thus, the resulting system~$(b_I^n)_{(n,I)\in \mathcal{D}_{\omega}}$ almost diagonalizes the operator~$T$. For the first step, we use the following finite-dimensional result from~\cite{MR4884827}.

\begin{lem}\label{lem:diagonalization-finite-dimensional}
  Let~$Y$ be a Haar system Hardy space, and let~$\Gamma,\eta > 0$. Moreover, let $n,N\in \mathbb{N}_0$ be chosen so that
  \begin{equation}\label{eq:6}
    N \ge N_0(n,\Gamma,\eta) := 21(n+1) + \Bigl\lfloor 4 \log_2\Bigl( \frac{\Gamma}{\eta} \Bigr) \Bigr\rfloor.
  \end{equation}
  Then for every linear operator $T\colon Y_N\to Y_N$ with~$\|T\|\le \Gamma$, there exists a finite faithful Haar system $(\hat{h}_I)_{I\in \mathcal{D}_{\le n}}$ with frequencies less than or equal to $N_0(n,\Gamma,\eta)$ as well as numbers $\alpha_0,\alpha_1,\dots,\alpha_n\in \mathcal{A}(\mathrm{diag}(T))$ such that the following holds: Put
  \begin{equation*}
    d_I = \frac{\langle \hat{h}_I, T \hat{h}_I \rangle}{|I|},\qquad I\in \mathcal{D}_{\le n},
  \end{equation*}
  then
  \begin{enumerate}[(i)]
    \item $|\langle \hat{h}_I, T \hat{h}_J \rangle|\le \eta |I|$ for all $I\ne J\in \mathcal{D}_{\le n}$.
    \item $|d_I|\le \|T\|$ for all $I\in \mathcal{D}_{\le n}$.
    \item $|d_I - \alpha_k| \le 8^{-n}\eta$ for all $I\in \mathcal{D}_k$, $0\le k\le n$.
  \end{enumerate}
\end{lem}
\begin{proof}
  This is proved in~\cite[Proposition~3.2]{MR4884827} (where we put~$\delta = 1$). The existence of the system~$(\hat{h}_I)_{I\in \mathcal{D}_{\le n}}$ and the numbers~$\alpha_0,\dots,\alpha_n$ is not stated explicitly, but the proof in~\cite{MR4884827} proceeds by constructing a system with the stated properties (and subsequently using the associated operators~$\hat{A}, \hat{B}$ to obtain a factorization result). Note that by \cite[equation~(3.15)]{MR4884827}, the numbers~$\alpha_k$ can be chosen as
  \begin{equation*}
    \alpha_k = 2^{-(m+k)} \sum_{K\in \mathcal{D}_{m+k}} \frac{\langle h_K, T h_K \rangle}{|K|} \in \mathcal{A}(\operatorname{diag}(T)),
    \qquad k=0,\dots,n,
  \end{equation*}
  where~$m\le N - n$ is a natural number.
\end{proof}

In order to apply this result in our setting, we will use the isometric isomorphisms $J_n\colon Y_n\to H_{\omega}^n$ defined in \Cref{rem:isometric-isomorphism-Jn}. Our construction yields a faithful Haar system~$(b_I^n)_{(n,I)\in \mathcal{D}_{\omega}}$ in~$Y_{\omega}$ whose properties are summarized in the following proposition.

\begin{pro}\label{pro:diagonalization}
  Let $Y$ be a Haar system Hardy space, and assume that the sequence of standard Rademacher functions~$(r_n)_{n=0}^{\infty}$ is weakly null in~$Y$. Let~$T\colon Y_{\omega}\to Y_{\omega}$ be a bounded linear operator, and let $\eta_{n,m} > 0$ for $n,m \in \mathbb{N}_0$. Then there exists a faithful Haar system~$(b_I^n)_{(n,I)\in \mathcal{D}_{\omega}}$ in~$Y_{\omega}$ such that $| \langle b_I^n, T b_J^m \rangle | \le \eta_{n,m}$ for all $(n,I)\ne (m,J)\in \mathcal{D}_{\omega}$.
  Moreover, the numbers
          \begin{equation*}
            d_I^n = \frac{\langle b_I^n, T b_I^n \rangle}{|I|},\qquad (n,I)\in \mathcal{D}_{\omega}
          \end{equation*}
  have the following properties:
  \begin{enumerate}[(i)]
    \item $|d_I^n|\le \|T\|$ for all $(n,I)\in \mathcal{D}_{\omega}$.
    \item For every $n\in \mathbb{N}_0$, there exist scalars $\alpha_0^n,\alpha_1^n,\dots,\alpha_n^n\in \mathcal{A}(\operatorname{diag}(T))$ such that the inequality $|d_I^n - \alpha_k^n| \le 8^{-n} \eta_{n,n}$ holds for all $I\in \mathcal{D}_k$, $0\le k\le n$.
  \end{enumerate}
\end{pro}
\begin{proof}
  Let~$T\colon Y_{\omega}\to Y_{\omega}$ be a bounded linear operator. For $n\in \mathbb{N}_0$, let $\eta_n > 0$ be chosen such that
  \begin{equation}\label{eq:11}
    \eta_n \le \min(\eta_{n,m}, \eta_{m,n})\qquad \text{for all } 0\le m\le n.
  \end{equation}
  Moreover, for $n\in \mathbb{N}_0$, put $N_0(n) := N_0(n,\|T\|,\eta_n)$, as defined in~\eqref{eq:6}.
  We will now choose numbers~$N(n)$, collections~$(\mathcal{B}_I^n)_{I\in \mathcal{D}_{\le n}}$ and signs~$(\theta_K^n)_{K\in \mathcal{D}}$ by induction on~$n$, which will then be used to define finite faithful Haar systems $(\hat{h}_I^n)_{I\in \mathcal{D}_{\le n}}$, $n\in \mathbb{N}_0$. Let~$n\in \mathbb{N}_0$, and assume that for every~$0\le m\le n-1$, we have already chosen~$N(m)$ and constructed collections $\mathcal{B}_I^m\subset \mathcal{D}_{N(m)}$ for $I\in \mathcal{D}_{\le m}$ and signs $(\theta_K^m)_{K\in \mathcal{D}}\in \{ \pm 1 \}^{\mathcal{D}}$. These collections and signs determine the functions
  \begin{equation*}
    \hat{h}_I^m = \sum_{K\in \mathcal{B}_I^m} \theta_K^m h_K
    \quad \text{and} \quad
    b_I^m = J_{N(m)} \hat{h}_I^m,
    \qquad
    I\in \mathcal{D}_{\le m},\ 0\le m\le n-1.
  \end{equation*}
  We now choose~$N(n)$ to be any integer with $N(n)\ge N_0(n)$. If~$n \ge 1$, then by increasing $N(n)$ even further, we can ensure that the following additional conditions are satisfied:
  \begin{enumerate}[(i)]
    \item $N(n) > N(n-1)$.
    \item For every function~$f\in H_{\omega}$ of the form $f = \sum_{K\in \mathcal{B}} \varepsilon_Kh_K^{N(n)}$, where~$\mathcal{B}\subset \mathcal{D}_{\le N_0(n)}$ and~$(\varepsilon_K)_{K\in \mathcal{B}}\in \{ \pm 1 \}^{\mathcal{B}}$, we have
    \begin{equation}\label{eq:8}
      |\langle f, T b_J^m \rangle| \le \eta_n
      \quad \text{and}\quad
      |\langle b_J^m, T f \rangle| \le \eta_n
    \end{equation}
      for all  $J\in \mathcal{D}_{\le m}$ and $0\le m\le n-1$.
  \end{enumerate}
  The last condition states that the functions~$f$ of the specified form almost annihilate all previously constructed elements $Tb_I^m\in Y_{\omega}$ and $T^{*}b_I^m\in Y_{\omega}^{*}$.
  To see that~\eqref{eq:8} is indeed satisfied if~$N(n)$ is chosen large enough, we exploit that for every fixed $K\in \mathcal{D}$, according to \Cref{lem:haar-weakly-null}, the sequence $(h_K^m)_{m=0}^{\infty}$ is weakly null in~$Y_{\omega}$ and weak* null in~$Y_{\omega}^{*}$. This yields the claimed result since there are only finitely many choices of~$\mathcal{B}$ and~$(\varepsilon_K)_{K\in \mathcal{B}}$ and finitely many functions~$b_J^m$ which have already been constructed.

  After choosing~$N(n)$, we consider the operator $T_{N(n)} = J_{N(n)}^{-1}P_{N(n)}TJ_{N(n)}$ (see \Cref{rem:isometric-isomorphism-Jn}). We apply \Cref{lem:diagonalization-finite-dimensional} to $T_{N(n)}\colon Y_{N(n)}\to Y_{N(n)}$, putting $\Gamma = \|T\|$ and $\eta = \eta_n$, to obtain a faithful Haar system
  \begin{equation*}
    \hat{h}_I^n = \sum_{K\in \mathcal{B}_I^n} \theta_K^n h_K\in Y_{N_0(n)}\subset Y_{N(n)}, \qquad I\in \mathcal{D}_{\le n}
  \end{equation*}
  for certain collections~$\mathcal{B}_I^n\subset \mathcal{D}_{\le N_0(n)}$, $I\in \mathcal{D}_{\le n}$, and signs~$(\theta_K^n)_{K\in \mathcal{D}}$. By applying the isometric isomorphism~$J_{N(n)}\colon Y_{N(n)}\to H_{\omega}^{N(n)}\subset Y_{\omega}$, we obtain the corresponding system
  \begin{equation}\label{eq:10}
    b_I^n = \sum_{K\in \mathcal{B}_I^n} \theta_K^n h_K^{N(n)}\in H_{\omega}^{N(n)},\qquad I\in \mathcal{D}_{\le n}.
  \end{equation}
  Moreover, we obtain numbers
  \begin{equation*}
    \alpha_0^n,\alpha_1^n,\dots,\alpha_n^n\in \mathcal{A}(\operatorname{diag}(T_{N(n)}))\subset \mathcal{A}(\operatorname{diag}(T)).
  \end{equation*}
  By \Cref{lem:diagonalization-finite-dimensional}, our newly constructed system has the following properties: Put
  \begin{equation*}
    d_I^n = \frac{\langle b_I^n, T b_I^n \rangle}{|I|} = \frac{\langle \hat{h}_I^n, T_{N(n)} \hat{h}_I^n \rangle}{|I|},\qquad I\in \mathcal{D}_{\le n}
  \end{equation*}
  (where we used \Cref{rem:isometric-isomorphism-Jn}). Then we have:
  \begin{enumerate}[(i)]
    \item $|\langle b_I^n, T b_J^n \rangle| = |\langle \hat{h}_I^n, T_{N(n)} \hat{h}_J^n \rangle|\le \eta_n |I|\le \eta_{n,n}$ for all $I\ne J\in \mathcal{D}_{\le n}$.
    \item $|d_I^n|\le \|T_{N(n)}\|\le \|T\|$ for all $I\in \mathcal{D}_{\le n}$.
    \item $|d_I^n - \alpha_k^n| \le 8^{-n}\eta_n\le 8^{-n}\eta_{n,n}$ for all $I\in \mathcal{D}_k$, $0\le k\le n$.
  \end{enumerate}
  Thus, it only remains to prove that
  \begin{equation}\label{eq:9}
    |\langle b_I^n, Tb_J^m \rangle| \le \eta_{n,m}
    \qquad \text{and}\qquad
    |\langle b_J^m, Tb_I^n \rangle| \le \eta_{m,n}
  \end{equation}
  hold for all $I\in \mathcal{D}_{\le n}$ and $J\in \mathcal{D}_{\le m}$, $0\le m\le n-1$. But since~$\mathcal{B}_I^n\subset \mathcal{D}_{\le N_0(n)}$ for all $I\in \mathcal{D}_{\le n}$, by the definition of~$b_I^n$ in~\eqref{eq:10}, we can apply~\eqref{eq:8} to $f = b_I^n$. Together with our choice of~$\eta_n$ in~\eqref{eq:11}, this proves~\eqref{eq:9}.
\end{proof}

Using \Cref{pro:diagonalization}, we can now prove our main result of this section.

\begin{thm}\label{thm:diagonalization}
  Let $Y$ be a Haar system Hardy space, and assume that the sequence of standard Rademacher functions~$(r_n)_{n=0}^{\infty}$ is weakly null in~$Y$. Let $T\colon Y_{\omega}\to Y_{\omega}$ be a bounded linear operator, and let~$\eta > 0$ and~$\eta_n > 0$ for all $n\in \mathbb{N}_0$. Then there exists a bounded linear operator $D\colon Y_{\omega}\to Y_{\omega}$ which is diagonal with respect to~$(h_I^n)_{(n,I)\in \mathcal{D}_{\omega}}$ such that $D$ projectionally factors through~$T$ with constant~$1$ and error~$\eta$. Moreover, the entries~$(d_I^n)_{(n,I)\in \mathcal{D}_{\omega}}$ of~$D$ have the following properties:
  \begin{enumerate}[(i)]
    \item\label{thm:diagonalization:i} $|d_I^n|\le \|T\|$ for all $(n,I)\in \mathcal{D}_{\omega}$.
    \item\label{thm:diagonalization:ii} For every~$n\in \mathbb{N}_0$, there exist scalars $\alpha_0^n,\alpha_1^n,\dots,\alpha_n^n\in \mathcal{A}(\operatorname{diag}(T))$ such that $|d_I^n -\nolinebreak \alpha_k^n| \le 8^{-n}\eta_n$ for all $I\in \mathcal{D}_k$, $0\le k\le n$.
  \end{enumerate}
  In particular, if~$T$ has $\delta$-large positive diagonal for some~$\delta > 0$, then we have $\alpha_k^n\ge \delta$ for all $0\le k\le n$ and hence, by~\eqref{thm:diagonalization:ii}, $d_I^n \ge \delta - 8^{-n}\eta_n$ for all~$(n,I)\in \mathcal{D}_{\omega}$.
\end{thm}
\begin{proof}
  Fix an operator~$T$ and~$\eta > 0$. Then we can find numbers $\eta_{n,m} > 0$ for $n,m \in \mathbb{N}_0$ such that
  \begin{equation}\label{eq:5}
    \sum_{(n,I),\, (m,J)\in \mathcal{D}_{\omega}} \frac{\eta_{n,m}}{|I| |J|} \le \frac{\eta}{2}
    \qquad \text{and} \qquad
    \eta_{n,n}\le \eta_n,\quad n\in \mathbb{N}_0.
  \end{equation}
  Let $(b_I^n)_{(n,I)\in \mathcal{D}_{\omega}}$ be the faithful Haar system in~$Y_{\omega}$ obtained from \Cref{pro:diagonalization}. Define $D\colon H_{\omega}\to H_{\omega}$ as the linear extension of $D h_I^n = d_I^nh_I^n$, $(n,I)\in \mathcal{D}_{\omega}$, where
  \begin{equation*}
    d_I^n = \frac{\langle b_I^n, T b_I^n \rangle}{|I|},
    \qquad (n,I)\in \mathcal{D}_{\omega}.
  \end{equation*}
  By \Cref{thm:operators-A-B}, the operators $\hat{A},\hat{B}\colon Y_{\omega}\to Y_{\omega}$ associated with the system~$(b_I^n)_{(n,I)\in \mathcal{D}_{\le n}}$ satisfy $\|\hat{A}\| = \|\hat{B}\| = 1$ and~$\hat{A}\hat{B} = I_{Y_{\omega}}$.
  For every $(m,J)\in \mathcal{D}_{\omega}$, we have
  \begin{align*}
    \|(\hat{A}T\hat{B} - D)h_J^m\|_Y
    &= \Bigl\| \sum_{(n,I)\in \mathcal{D}_{\omega}} \frac{\langle b_I^n, T b_J^m \rangle}{|I|}h_I^n - \frac{\langle b_J^m, T b_J^m \rangle}{|J|}h_J^m \Bigr\|_Y\\
    &\le \sum_{\substack{(n,I)\in \mathcal{D}_{\omega}\\ (n,I)\ne (m,J)}} \frac{|\langle b_I^n, T b_J^m \rangle|}{|I|} \|h_J^m\|_Y
    \le \sum_{\substack{(n,I)\in \mathcal{D}_{\omega}\\ (n,I)\ne (m,J)}} \frac{\eta_{n,m}}{|I|}.
  \end{align*}
  Now let $x = \sum_{(m,J)\in \mathcal{D}_{\omega}} a_J^mh_J^m \in H_{\omega}$. Then, applying \Cref{pro:HS-1d}~\eqref{pro:HS-1d:i} and \Cref{pro:HSHS-1d}~\eqref{pro:HSHS-1d:disjointly-supp}, and using that~$(h_J^m)_{(m,J)\in \mathcal{D}_{\omega}}$ is a monotone Schauder basis of~$Y_{\omega}$, we have $|J||a_J^m| = \|a_J^mh_J^m\|_{L^1}\le \|a_J^mh_J^m\|_Y \le  2\|x\|_Y$ for all~$(m,J)\in \mathcal{D}_{\omega}$. This implies that
  \begin{align*}
    \|(\hat{A}T\hat{B} - D)x\|_Y
    &\le \sum_{(m,J)\in \mathcal{D}_{\omega}} |a_J^m| \|(\hat{A}T\hat{B} - D) h_J^m\|_Y\\
    &\le 2 \|x\|_Y \sum_{(m,J)\in \mathcal{D}_{\omega}} \sum_{\substack{(n,I)\in \mathcal{D}_{\omega}\\ (n,I)\ne (m,J)}} \frac{\eta_{n,m}}{|I| |J|}\le \eta \|x\|_Y,
  \end{align*}
  where we used~\eqref{eq:5} in the last step. Properties~\eqref{thm:diagonalization:i} and~\eqref{thm:diagonalization:ii} of this theorem follow from the corresponding statements in \Cref{pro:diagonalization}. Thus, the continuous extension of~$D$ has the desired properties.
\end{proof}

\begin{rem}
  In the following, we describe an alternative way to construct an almost faithful Haar system~$(b_I^n)_{(n,I)\in \mathcal{D}_{\omega}}$ in~$Y_{\omega}$ which satisfies the properties stated in \Cref{pro:diagonalization} and hence yields a diagonalization result. This approach is based on the probabilistic arguments first used by R.~Lechner~\cite{MR3990955}. Recall that a refinement of these probabilistic techniques was recently applied by the first author and P.~Motakis~\cite{MR4839586} to prove factorization results in the Bourgain–Rosenthal–Schechtman $R_{\omega}^p$~space, $1<p< \infty$, and by the second author~\cite{MR4884827} to prove finite-dimensional, quantitative factorization results in Haar system Hardy spaces.
  However, to combine these approaches and apply them in our setting, we need the additional assumption that $\lim_{n\to \infty}\|\chi_{[0,2^{-n})}\|_Y  = \lim_{n\to \infty}\|\chi_{[0,2^{-n})}\|_{Y^{*}}= 0$ (which excludes, for example, the $L^1$, $H^1$, $L^{\infty}$ and $SL^{\infty}$ norm).

  First, we fix a strictly increasing function~$m\colon \mathbb{N}_0\to \mathbb{N}_0$ (to be determined later), and we define $N(n) = m(n) + n$ for all $n\in \mathbb{N}_0$. Assume that~$b_J^k$ has already been constructed for all~$(k,J)\in \mathcal{D}_{\omega}$ with $k < n$. We will construct a randomized faithful Haar system~$(\hat{h}_I^n(\theta))_{I\in \mathcal{D}_{\le n}}$ in~$Y_{N(n)}$ with frequencies $m(n), m(n) + 1, \dots, N(n)$. To this end, let $\theta = (\theta_K)_{K\in \mathcal{D}_{\le N(n)}}$ be chosen uniformly at random from~$\{ \pm 1 \}^{\mathcal{D}_{\le N(n)}}$,
  and denote the corresponding uniform measure, expected value and variance by $\mathbb{P}$, $\mathbb{E}$, and~$\mathbb{V}$, respectively.
  Now, as in \cite{MR4884827}, we recursively define
  \begin{equation*}
    \hat{h}_I^n(\theta) = \sum_{K\in \mathcal{B}_I(\theta)} \theta_K h_K,\qquad I\in \mathcal{D}_{\le n},\ \theta\in \{ \pm 1 \}^{\mathcal{D}_{\le N(n)}}
  \end{equation*}
  where $\mathcal{B}_{[0,1)}(\theta) = \mathcal{D}_{m(n)}$ and
  \begin{equation}\label{eq:12}
    \mathcal{B}_{I^{\pm }}(\theta) = \bigl\{ K\in \mathcal{D}_{m(n) + k + 1} : K\subset [\hat{h}_I^n(\theta) = \pm 1]\bigr\},\quad I\in \mathcal{D}_k,\ k=0,\dots,n-1.
  \end{equation}
  Moreover, we define the corresponding functions $b_I^n(\theta) = J_{N(n)} \hat{h}_I^n(\theta)\in H_{\omega}^{N(n)}$, $I\in \mathcal{D}_{\le n}$.
  Note that in contrast to the inductive randomization of single functions~$b_I^n$ in~\cite{MR4839586}, here all functions $b_I^n(\theta)$, $I\in \mathcal{D}_{\le n}$, are randomized at the same time.
  Now consider the following random variables:
  \begin{equation*}
    X_{I,J}(\theta) = \langle  b_{I}^{n}(\theta) , T b_{J}^{n}(\theta) \rangle, \qquad   I,J \in \mathcal{D}_{\leq n},
  \end{equation*}
  and, for $0\le k < n$ and $J \in \mathcal{D}_{\leq k}$, $I \in \mathcal{D}_{\leq n}$:
  \begin{equation*}
    Y_{b_J^k,I}(\theta) = \langle  b_J^k , Tb_I^n (\theta) \rangle \qquad \text{ and } \qquad W_{I,b_J^k}(\theta) = \langle b_I^n (\theta), Tb_J^k  \rangle.
  \end{equation*}
  It follows from \cite[Lemma~3.1]{MR4884827} and \Cref{rem:isometric-isomorphism-Jn} that
  \begin{equation*}
    \mathbb{V}(X_{I,J}) \leq 3\| T \|^2 2^{-m(n) / 2}, \qquad  I, J \in \mathcal{D}_{\leq n},
  \end{equation*}
  as well as $\mathbb{E}(X_{I,J}) = 0$ whenever $I\ne J$. The expected value of~$X_{I,I}$ is given by an explicit formula in~\cite{MR4884827}, which implies that it only depends on~$|I|$ (rather than~$I$), and it is always contained in $\mathcal{A}(\operatorname{diag}(T))$.
  Using the techniques from~\cite{MR4839586,MR4884827}, one can also show that for all $k < n$ and $J\in \mathcal{D}_{\le k}$, $I\in \mathcal{D}_{\le n}$,
  \begin{equation*}
    \mathbb{E}(Y_{b_J^k,I}) = \mathbb{E}(W_{I,b_J^k}) = 0
  \end{equation*}
  as well as
  \begin{equation*}
    \mathbb{V}(Y_{b_J^k,I}) \leq \|T\|^2 \|h_{[0,2^{-m(n)})}\|_Y \qquad \text{and} \qquad \mathbb{V}(W_{I,b_J^k}) \leq \|T\|^2 \|h_{[0,2^{-m(n)})}\|_{Y^{*}}.
  \end{equation*}
  Since we assumed that $\lim_{m\to \infty}\|\chi_{[0,2^{-m})}\|_Y  = \lim_{m\to \infty}\|\chi_{[0,2^{-m})}\|_{Y^{*}}= 0$,
  all variances can be made arbitrarily small by choosing $m(n)$ sufficiently large. Then, as in~\cite{MR4839586}, Chebyshev's inequality implies that for some choice of the signs~$\theta$, the above random variables are indeed very close to their respective expected values, which in turn yields a system~$(b_I^n)_{(n,I)\in \mathcal{D}_{\omega}}$ with the required properties for \Cref{pro:diagonalization}.
\end{rem}

\section{Reduction to a scalar operator}

In this section, we start with the diagonal and level-wise stabilized operator~$D\colon Y_{\omega}\to Y_{\omega}$ constructed in the previous section, and our goal is to further stabilize its entries and reduce~$D$ to a constant multiple of the identity~$cI_{Y_{\omega}}$. As before, our proof is based on a finite-dimensional result which is contained in the proof of~\cite[Proposition~4.1]{MR4884827}. We state this result below in a slightly modified and streamlined version, and for convenience, we provide a proof.

\begin{lem}\label{lem:stabilization-finite-dimensional}
  Let~$Y$ be a Haar system Hardy space, and let~$\Gamma,\eta > 0$. Moreover, let $n,N\in \mathbb{N}_0$ be chosen so that
  \begin{equation}\label{eq:13}
    N\ge N_1(n,\Gamma,\eta) := n \biggl\lceil \frac{2\Gamma}{\eta} \biggr\rceil + 1.
  \end{equation}
  Let $D\colon Y_N\to Y_N$ be a Haar multiplier with entries~$(d_I)_{I\in \mathcal{D}_{\le N}}$, and assume that there are scalars $\alpha_0,\alpha_1,\dots,\alpha_N\in [-\Gamma,\Gamma]$ and~$\xi > 0$ such that $|d_I - \alpha_k| \le \xi$ for all $I\in \mathcal{D}_k$, $0\le k\le N$. Then there exists a finite faithful Haar system~$(\hat{h}_I)_{I\in \mathcal{D}_{\le n}}$ with frequencies less than or equal to~$N$ and a constant~$c\in \{ \alpha_0,\alpha_1,\dots,\alpha_N \}$ such that the numbers
  \begin{equation*}
    \hat{d}_I = \frac{\langle \hat{h}_I, D \hat{h}_I \rangle}{|I|},\qquad I\in \mathcal{D}_{\le n}
  \end{equation*}
  satisfy $|\hat{d}_I - c| \le \eta + \xi$ for all $I\in \mathcal{D}_{\le n}$.
\end{lem}

\begin{proof}
  Divide the interval~$[-\Gamma,\Gamma]$ into $\lceil 2\Gamma/\eta \rceil$ subintervals of length at most~$\eta$. Then, by~\eqref{eq:13} and the pigeonhole principle, we can find integers $0\le k_0<k_1<\dots<k_n\le N$
  such that
  \begin{equation*}
    |\alpha_{k_i} - \alpha_{k_0}| \le \eta,
    \qquad i=0,\dots,n.
  \end{equation*}
  Hence, after putting~$c = \alpha_{k_0}$, by our assumption on the numbers~$\alpha_k$, we have
  \begin{equation}\label{eq:14}
    |d_K - c| \le \eta + \xi,\qquad K\in \mathcal{D}_{k_0}\cup \dots\cup \mathcal{D}_{k_n}.
  \end{equation}
  Now let~$(\hat{h}_I)_{I\in \mathcal{D}_{\le n}}$ be a finite faithful Haar system with frequencies~$k_0,\dots,k_n$. Then the numbers~$\hat{d}_I$ defined above satisfy
  \begin{equation*}
    \hat{d}_I = \frac{\langle \hat{h}_I, D \hat{h}_I \rangle}{|I|} = \sum_{K\in \mathcal{B}_I} d_K \frac{|K|}{|I|},\qquad I\in \mathcal{D}_{\le n}.
  \end{equation*}
  Together with~\eqref{eq:14} and the fact that~$|\mathcal{B}_I^{*}| = |I|$ for all~$I\in \mathcal{D}_{\le n}$, this yields
  \begin{equation*}
    |\hat{d}_I - c| \le \eta + \xi,\qquad I\in \mathcal{D}_{\le n}.\qedhere
  \end{equation*}
\end{proof}

\begin{pro}\label{pro:block-stabilization}
  Let $Y$ be a Haar system Hardy space, and let $D\colon Y_{\omega}\to Y_{\omega}$ be a bounded diagonal operator with entries~$(d_I^n)_{(n,I)\in \mathcal{D}_{\omega}}$. Moreover, let~$\Gamma,\eta > 0$, and assume that for every~$n\in \mathbb{N}_0$, there exist numbers $\alpha_0^n,\alpha_1^n,\dots,\alpha_n^n\in [-\Gamma,\Gamma]$ such that $|d_I^n -\nolinebreak \alpha_k^n| \le 8^{-n}\eta$ for all $I\in \mathcal{D}_k$, $0\le k\le n$.
  Then there exists a sequence of scalars $(c_n)_{n=0}^{\infty}$ and a bounded diagonal operator~$C\colon Y_{\omega}\to Y_{\omega}$ such that:
  \begin{enumerate}[(i)]
    \item $C$ projectionally factors through~$D$ with constant~$1$ and error~$16\eta$.
    \item $Ch_I^n = c_n h_I^n$ for all~$(n,I)\in \mathcal{D}_{\omega}$.
    \item For every~$n\in \mathbb{N}_0$, there exist integers $0\le k\le m$ such that~$c_n = \alpha_k^m$.
  \end{enumerate}
\end{pro}
\begin{proof}
  Put~$\eta_n = 8^{-n} \eta$ for~$n\in \mathbb{N}_0$. Choose a strictly increasing function $N\colon \mathbb{N}_0\to \mathbb{N}_0$ such that $N(n) \ge N_1(n,\Gamma,\eta_n)$ for every~$n\in \mathbb{N}_0$, where~$N_1$ is defined as in~\eqref{eq:13}.
  For every~$n\in \mathbb{N}_0$, consider the Haar multiplier $D_{N(n)}\colon Y_{N(n)}\to Y_{N(n)}$ defined as $D_{N(n)} = J_{N(n)}^{-1}D J_{N(n)}$ and recall that by \Cref{rem:isometric-isomorphism-Jn}, its entries are $(d_I^{N(n)})_{I\in \mathcal{D}_{\le N(n)}}$.
  Apply \Cref{lem:stabilization-finite-dimensional} to the operator~$D_{N(n)}$ (putting~$\eta = \xi = \eta_n\ge \eta_{N(n)}$) to obtain  a constant~$c_n\in \{ \alpha_0^{N(n)},\dots,\alpha_{N(n)}^{N(n)} \}$ and a finite faithful Haar system~$(\hat{h}_I^n)_{I\in \mathcal{D}_{\le n}}$ in~$Y_{N(n)}$ with the stated property. Let~$(b_I^n)_{I\in \mathcal{D}_{\le n}}$ be the corresponding system in~$H_{\omega}^{N(n)}$ given by~$b_I^n = J_{N(n)} \hat{h}_I^n$ for all~$I\in \mathcal{D}_{\le n}$. Then, according to \Cref{lem:stabilization-finite-dimensional}, the numbers
  \begin{equation*}
    \hat{d}_I^n
    = \frac{\langle b_I^n, D b_I^n \rangle}{|I|}
    = \frac{\langle \hat{h}_I^n, D_{N(n)} \hat{h}_I^n \rangle}{|I|},
    \qquad I\in \mathcal{D}_{\le n}
  \end{equation*}
  satisfy
  \begin{equation}\label{eq:16}
    |\hat{d}_I^n - c_n| \le 2\eta_n = 8^{-n-1}\cdot 16\eta,\qquad I\in \mathcal{D}_{\le n}.
  \end{equation}

  Now let $\hat{A},\hat{B}\colon Y_{\omega}\to Y_{\omega}$ denote the operators associated with the system~$(b_I^n)_{(n,I)\in \mathcal{D}_{\le n}}$, as defined in \Cref{thm:operators-A-B}. We know that $\|\hat{A}\| = \|\hat{B}\| = 1$ and~$\hat{A}\hat{B} = I_{Y_{\omega}}$. Then~$\widehat{D} = \hat{A}D\hat{B}$ is also a bounded diagonal operator on~$Y_{\omega}$, and its entries are~$(\hat{d}_I^n)_{(n,I)\in \mathcal{D}_{\omega}}$. Finally, let~$C\colon Y_{\omega}\to Y_{\omega}$ denote the diagonal operator with~$Ch_I^n = c_nh_I^n$ for all~$(n,I)\in \mathcal{D}_{\omega}$. Then, by~\eqref{eq:16} and \Cref{lem:diagonal-operator-bounded}, we have~$\|C - \hat{A}D\hat{B}\| = \|C - \widehat{D}\| \le 16\eta$.
\end{proof}

\begin{pro}\label{pro:stabilization-cluster-point}
  Let~$Y$ be a Haar system Hardy space, let~$\eta > 0$, and let~$C\colon Y_{\omega}\to Y_{\omega}$ be a bounded diagonal operator such that~$C h_I^n = c_n h_I^n$ for all $(n,I)\in \mathcal{D}_{\omega}$, where~$(c_n)_{n=0}^{\infty}$ is a sequence of scalars. Then for every cluster point~$c$ of~$(c_n)_{n=0}^{\infty}$, the operator~$cI_{Y_{\omega}}$ projectionally factors through~$C$ with constant~$1$ and error~$\eta$.
\end{pro}
\begin{proof}
  Let~$c$ be a cluster point of the sequence~$(c_n)_{n=0}^{\infty}$. Then we can find a strictly increasing function $N\colon \mathbb{N}_0\to \mathbb{N}_0$ such that $|c_{N(n)} - c| \le 8^{-1}4^{-n}\eta$ for all~$n\in \mathbb{N}_0$. Now, for~$(n,I)\in \mathcal{D}_{\omega}$, put
  \begin{equation*}
    b_I^n = h_I^{N(n)} \in H_{\omega}^{N(n)}.
  \end{equation*}
  Let~$\hat{A},\hat{B}\colon Y_{\omega}\to Y_{\omega}$ denote the operators associated with the system~$(b_I^n)_{(n,I)\in \mathcal{D}_{\omega}}$, then we know from \Cref{thm:operators-A-B} that~$\|\hat{A}\| = \|\hat{B}\| = 1$ and~$\hat{A}\hat{B} = I_{Y_{\omega}}$. Moreover, $\hat{A}C\hat{B}$ is a diagonal operator with~$\hat{A}C\hat{B} h_I^n = c_{N(n)}h_I^n$ for all~$(n,I)\in \mathcal{D}_{\omega}$. Since~$|c_{N(n)} - c|\le 8^{-1}4^{-n}\eta$ for all $n\in \mathbb{N}_0$, \Cref{lem:diagonal-operator-bounded} implies that~$\|c I_{Y_{\omega}} - \hat{A}C\hat{B}\|\le \eta$.
\end{proof}

By combining the above results with \Cref{thm:diagonalization}, we obtain the following theorem.

\begin{thm}\label{thm:constant-multiple-of-identity}
  Let~$Y$ be a Haar system Hardy space, and assume that the sequence of standard Rademacher functions~$(r_n)_{n=0}^{\infty}$ is weakly null in~$Y$. Let~$T\colon Y_{\omega}\to Y_{\omega}$ be a bounded linear operator, and let~$\eta > 0$. Then there exists a scalar~$c\in \overline{\mathcal{A}(\operatorname{diag}(T))}$ such that $cI_{Y_{\omega}}$ projectionally factors through~$T$ with constant~$1$ and error~$\eta$. If~$T$ has $\delta$-large positive diagonal with respect to~$(h_I^n)_{(n,I)\in \mathcal{D}_{\omega}}$ for some~$\delta > 0$, then~$c \ge \delta$.
\end{thm}
\begin{proof}
  Let~$T\colon Y_{\omega}\to Y_{\omega}$ be a bounded linear operator, and let~$\eta > 0$. First, apply \Cref{thm:diagonalization} (setting~$\eta_n = \eta$ for all $n\in \mathbb{N}_0$) to obtain a diagonal operator~$D\colon Y_{\omega}\to Y_{\omega}$ with entries~$(d_I^n)_{(n,I)\in \mathcal{D}_{\omega}}$ such that~$D$ projectionally factors through~$T$ with constant~$1$ and error~$\eta$, as well as scalars~$\alpha_k^n\in \mathcal{A}(\operatorname{diag}(T))$, $0\le k\le n$, such that~$|d_I^n - \alpha_k^n|\le 8^{-n}\eta$ for all $I\in \mathcal{D}_k$, $0\le k\le n$. Next, put~$\Gamma = \|T\|$ and note that~$\mathcal{A}(\operatorname{diag}(T))\subset [-\Gamma,\Gamma]$ (see \Cref{rem:isometric-isomorphism-Jn}). Apply \Cref{pro:block-stabilization} to obtain a sequence of scalars~$(c_n)_{n=0}^{\infty}$ in~$\{ \alpha_k^m : k,m \in \mathbb{N}_0,\, 0\le k\le m \}$ defining a diagonal operator~$C\colon Y_{\omega}\to Y_{\omega}$ with~$Ch_I^n = c_nh_I^n$ for all~$(n,I)\in \mathcal{D}_{\omega}$ such that~$C$ projectionally factors through~$D$ with constant~$1$ and error~$16\eta$. Finally, let~$c$ be a cluster point of the sequence~$(c_n)_{n=0}^{\infty}$, then by \Cref{pro:stabilization-cluster-point}, $cI_{Y_{\omega}}$~factors through~$C$ with constant~$1$ and error~$\eta$.

By combining all these factorizations using \Cref{rem:transitivity}, we obtain that~$cI_{Y_{\omega}}$ projectionally factors through~$T$ with constant~$1$ and error~$18\eta$. If~$T$ has~$\delta$-large positive diagonal for some~$\delta > 0$, then we have~$\mathcal{A}(\operatorname{diag}(T))\subset [\delta,\infty)$ and hence~$\alpha_k^n\ge \delta$ for all~$0\le k\le n$, which implies that $c_n\ge \delta$ for all~$n\in \mathbb{N}_0$ and thus~$c\ge \delta$.
\end{proof}

As a consequence, we obtain that the space~$Y_{\omega}$ has the primary factorization property, which is the first statement of the following theorem. Another consequence is that the identity on~$Y_{\omega}$ factors through all operators with large \emph{positive} diagonal. If the basis~$(h_I^n)_{(n,I)\in \mathcal{D}_{.\omega}}$ is unconditional, then this is equivalent to the factorization property. In general, however, the factorization property requires an additional reduction step, which will be provided in the next section.

\begin{thm}\label{thm:primary-and-pos-fact-prop}
Let~$Y$ be a Haar system Hardy space, and assume that the sequence of standard Rademacher functions~$(r_n)_{n=0}^{\infty}$ is weakly null in~$Y$. Then the following assertions are true:
\begin{enumerate}[(i)]
  \item\label{thm:primary-and-pos-fact-prop:i} For every bounded linear operator $T \colon Y_{\omega} \to Y_{\omega}$, the identity~$I_{Y_{\omega}}$ factors either through~$T$ or through~$I_{Y_{\omega}} - T$ with constant~$2^+$.
  \item\label{thm:primary-and-pos-fact-prop:ii} For every~$\delta > 0$ and every bounded linear operator $T \colon Y_{\omega} \to Y_{\omega}$ with $\delta$-large \emph{positive} diagonal, the identity~$I_{Y_{\omega}}$ factors through~$T$ with constant~$(1/\delta)^+$.
\end{enumerate}
\end{thm}
\begin{proof}
  Let $T \colon Y_{\omega} \to Y_{\omega}$ be an arbitrary bounded linear operator and let $0 < \eta < \frac{1}{2}$. By \Cref{thm:constant-multiple-of-identity}, there exists a scalar $c\in \overline{\mathcal{A}(\operatorname{diag}(T))}$ such that $cI_{Y_{\omega}}$ projectionally factors through~$T$ with constant~$1$ and error~$\eta$, i.e., there exist bounded linear operators $A, B \colon Y_{\omega}  \to Y_{\omega}$ such that $\| A|| \cdot \| B \| \leq 1$ and $\| ATB -  cI_{Y_{\omega}} \| \le \eta$.
  We consider two cases for the scalar~$c$. If $c \geq 1/2$, then
  \begin{align*}
    \| c^{-1} A T B - I_{Y_{\omega}} \| \le 2\eta <1.
  \end{align*}
  Hence, the operator $c^{-1} A T B$ is invertible with $\|( c^{-1} A T B)^{-1} \| \leq 1/(1 - 2 \eta)$. Putting  $L = ( c^{-1} A T B )^{-1} c^{-1} A$ and $R = B$, we have $LTR = I_{Y_{\omega}}$ and
  \begin{align*}
    \| L \| \cdot \| R \| \leq  \frac{2}{1 - 2\eta}.
  \end{align*}
  On the other hand, assume that $c < 1/2$. Then, since $1-c > 1/2$ and
  \begin{align*}
    \| A ( I_{Y_{\omega}} - T) B - (1- c) I_{Y_{\omega}} \| \le \eta,
  \end{align*}
  we achieve the same conclusion for $I_{Y_{\omega}}-T$ instead of $T$. This completes the proof of~\eqref{thm:primary-and-pos-fact-prop:i}.

  To prove~\eqref{thm:primary-and-pos-fact-prop:ii}, let $T \colon Y_{\omega} \to Y_{\omega}$ be a bounded linear operator with $\delta$-positive large diagonal for some $\delta > 0$. By \Cref{thm:constant-multiple-of-identity}, for every~$\eta > 0$, there exists a scalar $c\in \overline{\mathcal{A}(\operatorname{diag}(T))}$ with $c \geq \delta$ such that $cI_{Y_{\omega}}$ projectionally factors through~$T$ with constant~$1$ and error~$\eta$. Now, using $c \geq \delta$ and an analogous argument to the first case above (where $c \geq 1/2$), we obtain the stated result.
\end{proof}

\section{Reduction to positive diagonal}

In this section, we prove that for every bounded linear operator~$T$ on~$Y_{\omega}$ with large diagonal, there exists a bounded linear operator~$\widetilde{T}$ with large \emph{positive} diagonal such that $\tilde{T}$ factors through $T$ (\Cref{pro:positive-diagonal}). Finally, using this result together with \Cref{thm:primary-and-pos-fact-prop}~\eqref{thm:primary-and-pos-fact-prop:ii}, we obtain the factorization property of the basis $(h_I^n)_{(n,I)\in \mathcal{D}_{\omega}}$ of $Y_{\omega}$ (\Cref{thm:fact-property}).

Again, our proof is based on a finite-dimensional result from \cite{MR4884827}:
\begin{lem}\label{lem:positive-diagonal}
  Let $Y$ be a Haar system Hardy space, and let $\eta,\delta > 0$. Moreover, let $n,N\in \mathbb{N}_0$ be chosen so that
  \begin{equation}\label{eq:2}
    N \ge N_2(n,\eta) := 2n \biggl\lceil \frac{n}{\eta} + 1 \biggr\rceil 2^n.
  \end{equation}
  Then for every linear operator $T\colon Y_N\to Y_N$ with $\delta$-large diagonal with respect to the Haar system, there exists a finite almost faithful Haar system $(\tilde{h}_I)_{I\in \mathcal{D}_{\le n}}$ in~$Y_N$ such that
  \begin{equation*}
     \langle \tilde{h}_I, T \tilde{h}_I \rangle \ge \delta \|\tilde{h}_I\|_{L^2}^2 \text{ for all }I\in \mathcal{D}_{\le n} \quad \text{or}\quad  \langle \tilde{h}_I, T \tilde{h}_I \rangle \le -\delta \|\tilde{h}_I\|_{L^2}^2 \text{ for all } I\in \mathcal{D}_{\le n}.
  \end{equation*}
  Moreover, the Haar supports $\mathcal{B}_I$ of $\tilde{h}_I$ satisfy
  \begin{equation*}
    0\le \frac{|I|}{|\mathcal{B}_I^{*}|} - \frac{1}{\mu} \le \frac{\eta}{n \mu},\qquad I\in \mathcal{D}_{\le n},
  \end{equation*}
  where $\mu = |\mathcal{B}_{[0,1)}^{*}| \ge 1/2$.
\end{lem}
\begin{proof}
  This is shown in \cite[Lemma~6.1]{MR4884827} using a finite version of the Gamlen-Gaudet construction~\cite{MR0328575}; the properties of the system $(\tilde{h}_I)_{I\in \mathcal{D}_{\le N}}$ follow from~(5.3) in~\cite{MR4884827} and from the final estimates in the proof.
\end{proof}

Next, we prove that every bounded linear operator on~$Y_{\omega}$ with large diagonal can be reduced to a bounded linear operator with large \emph{positive} diagonal.

\begin{pro}\label{pro:positive-diagonal}
Let~$Y$ be a Haar system Hardy space, let $\delta,\eta > 0$, and let~$T\colon Y_{\omega}\to Y_{\omega}$ be a bounded linear operator with $\delta$-large diagonal with respect to~$(h_I^n)_{(n,I)\in \mathcal{D}_{\omega}}$. Then there exists another bounded linear operator~$\widetilde{T}\colon Y_{\omega}\to Y_{\omega}$ with $\delta$-large \emph{positive} diagonal such that~$\widetilde{T}$ factors through~$T$ with constant~$4+\eta$.
\end{pro}

\begin{proof}
  For $n\in \mathbb{N}_0$, put $\eta_n = 4^{-(n+2)}\eta$. Let $N\colon \mathbb{N}_0\to \mathbb{N}_0$ be a strictly increasing function such that $N(n)\ge N_2(n,\eta_n)$ for all $n\in \mathbb{N}_0$, where $N_2$ is as in~\eqref{eq:2}. For every $n\in \mathbb{N}_0$, consider the operator $T_{N(n)}\colon Y_{N(n)}\to Y_{N(n)}$ defined as $T_{N(n)} = J_{N(n)}^{-1}P_{N(n)}TJ_{N(n)}$, and recall that by \Cref{rem:isometric-isomorphism-Jn}, $T_{N(n)}$ has $\delta$-large diagonal. Hence, by \Cref{lem:positive-diagonal}, there exists a finite almost  faithful Haar system $(\tilde{h}_I^n)_{I\in \mathcal{D}_{\le n}}$ in~$Y_{N(n)}$ such that, after replacing~$T$ by~$-T$ if necessary, we have
  \begin{equation}\label{eq:19}
    \langle \tilde{h}_I^n, T_{N(n)} \tilde{h}_I^n \rangle \ge \delta \|\tilde{h}_I^n\|_{L^2}^2, \qquad I\in \mathcal{D}_{\le n}.
  \end{equation}
  Moreover, we have
  \begin{equation}\label{eq:4}
    0\le \frac{|I|}{|B_I^n|} - \frac{1}{\mu_{n}} \le \frac{\eta_n}{n \mu_n},\qquad I\in \mathcal{D}_{\le n},
  \end{equation}
  where $B_I^n = \operatorname{supp} \tilde{h}_I^n$ for all $I\in \mathcal{D}_{\le n}$ and $\mu_n = |B_{[0,1)}^n| \ge 1/2$. Now, for $(n,I)\in \mathcal{D}_{\omega}$, put $\tilde{b}_I^n = J_{N(n)} \tilde{h}_I^n\in Y_{\omega}$. Then $(\tilde{b}_I^n)_{(n,I) \in \mathcal{D}_{\omega}}$ is an almost faithful Haar system in $Y_{\omega}$ which satisfies the assumptions of~\Cref{thm:operators-A-B-almost-faithful} (this follows from~\eqref{eq:4} and \Cref{rem:isometric-isomorphism-Jn}). Thus, the associated operators $A,B\colon Y_{\omega}\to Y_{\omega}$ are well-defined and satisfy $AB = I_{Y_{\omega}}$ as well as $\|B\|\le 1$ and $\|A\|\leq 4 + \eta$. Finally, put $\widetilde{T} = ATB$. Then $\widetilde{T}$ factors through~$T$ with constant $4+\eta$ and error~$0$, and $\widetilde{T}$ has $\delta$-large \emph{positive} diagonal: For every $(n,I)\in \mathcal{D}_{\omega}$, using~\eqref{eq:19}, \Cref{rem:isometric-isomorphism-Jn} and the fact that $\|\tilde{b}_I^n\|_{L^2}^2 = \|\tilde{h}_I^n\|_{L^2}^2$, we have
  \begin{equation*}
    \frac{\langle h_I^n, \widetilde{T} h_I^n \rangle}{|I|}
    = \frac{\langle \tilde{b}_I^n, T \tilde{b}_I^n \rangle}{\|\tilde{b}_I^n\|_{L^2}^2}
    = \frac{\langle \tilde{h}_I^n, T_{N(n)} \tilde{h}_I^n \rangle}{\|\tilde{h}_I^n\|_{L^2}^2} \ge \delta.\qedhere
  \end{equation*}
\end{proof}

By combining \Cref{pro:positive-diagonal} with \Cref{thm:primary-and-pos-fact-prop}~\eqref{thm:primary-and-pos-fact-prop:ii}, we obtain the factorization property of the basis~$(h_I^n)_{(n,I)\in \mathcal{D}_{\omega}}$ of $Y_{\omega}$:

\begin{thm}\label{thm:fact-property}
  Let~$Y = X_0(\mathbf{r})$ be a Haar system Hardy space and suppose that the sequence of standard Rademacher functions $(r_n)_{n=0}^{\infty}$ is weakly null in~$Y$. Then for every~$\delta > 0$ and every bounded linear operator $T \colon Y_{\omega} \to Y_{\omega}$ with $\delta$-large diagonal with respect to~$(h_I^n)_{(n,I)\in \mathcal{D}_{\omega}}$, the identity~$I_{Y_{\omega}}$ factors through~$T$ with constant~$(4/\delta)^+$ (and with constant~$(1/\delta)^+$ if~$\mathbf{r}$ is independent).
\end{thm}

\noindent\textbf{Acknowledgments.}
The authors would like to thank Richard Lechner and Pavlos Motakis for many helpful comments and suggestions.

\bibliographystyle{abbrv}%
\bibliographystyle{plain}%
\bibliography{bibliography}%

\end{document}